\documentclass[pdflatex,sn-mathphys-num]{sn-jnl}% Math and Physical Sciences Numbered Reference Style
\usepackage{graphicx}%
\usepackage{multirow}%
\usepackage{amsmath,amssymb,amsfonts}%
\usepackage{amsthm}%
\usepackage{mathrsfs}%
\usepackage[title]{appendix}%
\usepackage{xcolor}%
\usepackage{textcomp}%
\usepackage{manyfoot}%
\usepackage{booktabs}%
\usepackage{algorithm}%
\usepackage{algorithmicx}%
\usepackage{algpseudocode}%
\usepackage{listings}%
\usepackage{xcolor}%  % keep as is (you can delete the duplicate earlier one if you like)

\makeatletter
\let\orcidlogo\relax   % forget the definition from sn-jnl.cls
\makeatother
\usepackage{orcidlink} % now orcidlink can define \orcidlogo without conflict

\theoremstyle{thmstyleone}%
\newtheorem{theorem}{Theorem}%  meant for continuous numbers
\newtheorem{lemma}[theorem]{Lemma}
\newtheorem{corollary}[theorem]{Corollary}
\newtheorem{proposition}[theorem]{Proposition}% 

\theoremstyle{thmstyletwo}%
\newtheorem{remark}{Remark}%

\theoremstyle{thmstylethree}%
\newtheorem{definition}{Definition}%

\renewcommand{\>}{\rangle}

\begin{document}

\title[Spectral Theory of Krein-Feller Type Operators and Applications in Stochastic Fractional Elliptic and Parabolic Equations]{Spectral Theory of Krein-Feller Type Operators and Applications in Stochastic Fractional Elliptic and Parabolic Equations}

%%=============================================================%%
%% GivenName	-> \fnm{Joergen W.}
%% Particle	-> \spfx{van der} -> surname prefix
%% FamilyName	-> \sur{Ploeg}
%% Suffix	-> \sfx{IV}
%% \author*[1,2]{\fnm{Joergen W.} \spfx{van der} \sur{Ploeg} 
%%  \sfx{IV}}\email{iauthor@gmail.com}
%%=============================================================%%
\author*[1,2]{\fnm{Kelvin} \sur{J.R. Almeida-Sousa}\orcidlink{0000-0002-9644-9819}}\email{kelvinjhonson.silva@kaust.edu.sa}

\author[1]{\fnm{Alexandre} \sur{B. Simas}\orcidlink{0000-0003-2562-2829}}\email{alexandre.simas@kaust.edu.sa}
\equalcont{These authors contributed equally to this work.}

% \author[1,2]{\fnm{Third} \sur{Author}}\email{iiiauthor@gmail.com}
% \equalcont{These authors contributed equally to this work.}

\affil*[1]{\orgdiv{Statistics Program, Computer, Electrical and Mathematical Science and Engineering Division}, \orgname{King Abdullah University of Science and Technology (KAUST)}, \orgaddress{\city{Thuwal}, \postcode{23955-6900}, \country{Saudi Arabia}}}

\affil[2]{\orgdiv{Department of Mathematics}, \orgname{Universidade Federal da Paraíba (UFPB)}, \orgaddress{\city{João Pessoa}, \postcode{58059-900}, \state{Paraíba}, \country{Brazil}}}

% \affil[3]{\orgdiv{Department}, \orgname{Organization}, \orgaddress{\street{Street}, \city{City}, \postcode{610101}, \state{State}, \country{Country}}}

%%==================================%%
%% Sample for unstructured abstract %%
%%==================================%%

\abstract{It has been shown that the space $C^{\infty}_{W,V}(\mathbb{T})$, introduced in Simas and Sousa (Potential Analysis, 2025), is the natural regularity space for solutions of the eigenvalue problem $\Delta_{W,V} u = \lambda u$ on the torus $\mathbb{T}$, where $\Delta_{W,V} = \frac{d^{+}}{dV}\frac{d^{-}}{dW}$ is the Krein Feller operator in the case where $W$ and $V$ are strictly increasing and right continuous (respectively left continuous), possibly with dense sets of discontinuities. In this work we provide conditions ensuring that every function in $C^{\infty}_{W,V}(\mathbb{T})$, which may be highly discontinuous, admits a series expansion that generalizes the classical Taylor expansion. A central feature of our approach is that all proofs are nonstandard, since classical analytical and spectral arguments cannot be adapted to this singular setting. Using these methods we characterize the eigenvectors of $\Delta_{W,V}$ in terms of generalized trigonometric functions and obtain an asymptotic lower bound for the associated eigenvalues. We also derive a sharp upper bound for the convergence exponent of these eigenvalues, and as a consequence we prove that $C^{\infty}_{W,V}(\mathbb{T})$ is a nuclear space. Further consequences include results on the asymptotic behavior of eigenvalues of compact operators and improvements in traceability. As a final application we establish existence results for generalized fractional stochastic and deterministic differential equations, as well as for parabolic stochastic partial differential equations acting on nuclear spaces.}

%%================================%%
%% Sample for structured abstract %%
%%================================%%


%%\pacs[MSC Classification]{35A01, 65L10, 65L12, 65L20, 65L70}

\maketitle

\section{Introduction}

The earliest practical use of generalized second-order operators is attributed to Feller~\cite{feller}. Following this seminal work, a wide range of developments have emerged in both theoretical and applied domains, leading to the widespread adoption of such operators in various fields. For example, the Krein–Feller operator $\frac{d}{dW}\left(\frac{d}{dx}\right)$ models the vibration of a string with non-uniform mass, where $W$ represents the cumulative mass distribution along the string (see, e.g., \cite{dym2008gaussian,Kats_1994}). More generally, when $W$ and $V$ induce non-atomic measures on an interval, the so-called geometric Krein–Feller operator $\frac{d}{dV}\left(\frac{d}{dW}\right)$ arises as a natural generalization of the Laplacian. This operator provides a robust framework for defining differential operators on irregular structures such as fractal sets. For instance, this occurs when $V(x) = x$ and $W$ is the distribution function of the Cantor measure on the interval $[0,1]$ (see, e.g., \cite{trig, uta, Freiberg_2005}). Geometric second-order operators of this type are also deeply connected with the generators of strongly continuous Markov processes and certain classes of inhomogeneous Markov processes (see, e.g., \cite{dynkin1965markov, keale}). The generalized derivative operator $D_{W}$ is also known in the literature as the Riemann–Stieltjes derivative, and several classical results from the theory of ordinary differential equations and real analysis have been extended to this setting. These developments highlight the important role played by first-order generalized differential operators (see, e.g., \cite{marcia, pouso}).

\subsection{A Basic Convention and the Space $C^{\infty}_{W,V}(\mathbb{T})$}

Throughout the literature, various hypotheses on $W$ have been considered in order to define $D_{W}$, and for each such choice, the local definition of $D_{W}$ may differ. However, to define this derivative operator on a broad class of functions, $D_{W}$ essentially reduces to the Radon–Nikodym derivative with respect to the measure $dW$ induced by $W$, which is fundamentally a non-local notion of derivative, despite the $dW$-almost everywhere existence of ``local'' derivatives (we use quotation marks because such derivatives may not exist at every point). The choice of a suitable local definition of $D_{W}$ and the regularity of $W$ can provide a balance between the complexity of the local structure and the theoretical tools available for the non-local version of $ D_{W} $. In this work, we adopt the local versions considered in \cite{keale}, where $W$ (resp. $ V $) is strictly increasing and right-continuous (resp. left-continuous) at every point of the domain. The local version of $D_{W}$ (resp. $D_{V}$) is denoted by $D^{-}_{W}$ (resp. $D^{+}_{V}$), the $W$-left-derivative (resp. $V$-right-derivative). For practical reasons and to emphasize certain advantages, we also use $ D^{-}_{W} $ (resp. $ D^{+}_{V} $) to denote the non-local version, referring to it as the weak $W$-left-derivative (resp. weak $V$-right-derivative).

Following this convention, the authors in \cite{keale} showed that the energetic space $H_{W,V}(\mathbb{T})$ (in the sense of Zeidler \cite{zeidler}) associated with the operator $\Delta_{W,V}:=D_{V}^{+}D_{W}^{-}$ can be viewed as a generalized Sobolev space on $\mathbb{T}$, which recovers the classical space $H^{1}(\mathbb{T})$ when $W(x)=V(x)=x$. The characterization of $H^{1}_{W,V}(\mathbb{T})$ as a generalized Sobolev space was made possible by introducing the space of test functions $C^{\infty}_{V,W}(\mathbb{T})$, which turns out to be the natural space where solutions to the problem
\begin{equation}\label{eigeneq}
	\begin{cases}
		-\Delta_{V,W}u=\lambda u;\\
		u\in \mathcal{D}_{V,W}(\mathbb{T}).
	\end{cases}
\end{equation}
belong. In (\ref{eigeneq}), the space $\mathcal{D}_{V,W}(\mathbb{T})\subset L^{2}_{V}(\mathbb{T})$ is the domain of $-D^{-}_{W}D^{+}_{V}$, which is the formal adjoint of the operator $\Delta_{W,V}$. Under the usual assumptions on $V$ and $W$, both $C^{\infty}_{W,V}(\mathbb{T})$ and $C^{\infty}_{V,W}(\mathbb{T})$ generalize the classical space $C^{\infty}(\mathbb{T})$ of test functions.

\subsection{Generalized Taylor Expansion in $C^{\infty}_{W,V}(\mathbb{T})$}

The first goal of this paper is to exhibit a class of elementary functions in $C^{\infty}_{W,V}(\mathbb{T})$ that admit a series representation which generalizes the classical Taylor expansion. More precisely, we derive sufficient conditions for a function $f\in C^{\infty}_{W,V}(\mathbb{T})$ to be represented as 
\begin{equation}\label{gen_taylor}
	f(x)=f(0)+\sum_{k=1}^{\infty} D^{(k)}_{W,V}f(0)F_{k}(x),
\end{equation}
where the series converges uniformly on $\mathbb{T}$, $D^{(k)}_{W,V}$ is the local high-order one-sided derivative operator, and $F_{k}(x):=F_{k}(x,x)$, with $F_{1}(x,s)=W(s)$ and 
\begin{equation*}
	F_{n}(x,s)=\begin{cases}
		\displaystyle\int_{[0,s)}\left[F_{n-1}(x,x)-F_{n-1}(x,\xi)\right]dV(\xi),\quad \text{n even};\\
		\displaystyle\int_{(0,s]}\left[F_{n-1}(x,x)-F_{n-1}(x,\xi)\right]dW(\xi),\quad \text{n odd}.
	\end{cases}
\end{equation*}
Note that the expansion above is centered at the origin of $\mathbb{T}$. This is intentional, as any point on $\mathbb{T}$ could serve as the origin, and thus we refer to it as a generalized Maclaurin expansion. Moreover, the choice of the torus $\mathbb{T}$ is not restrictive: the expansion can be formulated on any interval of $\mathbb{R}$. In particular, our result on the uniform convergence of the series in \eqref{gen_taylor} holds uniformly on compact subsets of $\mathbb{R}$ (see Theorem \ref{analyticalVW}). We emphasize that when $V(x) = W(x) = x$, we recover $F_{k}(x,x) = \frac{x^{k}}{k!}$, so the series reduces to the classical Maclaurin expansion. However, establishing conditions for the uniform convergence of the series in \eqref{gen_taylor} is much more subtle than in the classical case, since $W$ and $V$ may be highly discontinuous. Consequently, estimating the term $F_{k}(x,s)$ is a nontrivial problem, as there is no explicit formula available.

Beyond providing explicit examples of functions in $C^{\infty}_{W,V}(\mathbb{T})$ that can be effectively computed through ``elementary" terms $F_{k}$, this analytical representation yields several important consequences. In particular, it allows us to define the functions $C_{W,V}(\alpha, x)$, $S_{W,V}(\alpha, x)$, $C_{V,W}(\alpha, x)$, and $S_{V,W}(\alpha, x)$ for $\alpha \in \mathbb{R}$ and $x\in\mathbb{T}$ (see Section \ref{sect42}), which, for example, characterize any non-trivial solution $u$ of \eqref{eigeneq} as 
\begin{equation}\label{eigen_charact}
	u(x)=aC_{W,V}\left(\sqrt{\lambda},x\right)+\frac{b}{\sqrt{\lambda}}S_{W,V}\left(\sqrt{\lambda},x\right),
\end{equation}
for some $a,b\in\mathbb{R}$ with $(a,b)\neq (0,0)$ satisfying 
\begin{equation}\label{bnd_transposed}
	\begin{cases}
		a\left[C_{W,V}\left(\sqrt{\lambda},1\right)-1\right]+\frac{b}{\sqrt{\lambda}}S_{W,V}(\sqrt{\lambda},1)=0; \\
		\frac{a}{\sqrt{\lambda}}S_{V,W}(\sqrt{\lambda},1)-\frac{b}{\lambda}\left[C_{V,W}\left(\sqrt{\lambda},1\right)-1\right]=0,
	\end{cases}
\end{equation}
where the equality in \eqref{eigen_charact} holds up to a set of $dV$-measure zero on $\mathbb{T}$, and \eqref{bnd_transposed} reflects the cyclic boundary conditions. A basic consequence of this characterization is the recovery of regularity results for solutions of \eqref{eigeneq}.

\subsection{Spectral asymptotics of $\Delta_{W,V}$ and nuclearity of $C^{\infty}_{W,V}(\mathbb{T})$}

We have already emphasized that the space $C_{W,V}^{\infty}(\mathbb{T})$ is used as the test function space for defining weak generalized lateral derivatives. However, it is natural to ask to what extent this space plays the same foundational role as $C^{\infty}(\mathbb{T})$ in the classical theory. A core reason for the prominence of $C^{\infty}(\mathbb{T})$ is its nuclearity. If $C_{W,V}^{\infty}(\mathbb{T})$ shares this property, the two spaces may be considered theoretically indistinguishable within their respective frameworks. As we shall see in our section on applications, this property is crucial for determining the existence and uniqueness of solutions to evolution equations valued in the dual of a nuclear space. In the generalized setting, a natural approach is to construct nuclear spaces using high-order $W$-$V$ Sobolev spaces, denoted by $H^{n}_{W,V}(\mathbb{T}^d)$ for $n \in \mathbb{N}$, where $\mathbb{T}^d$ is the $d$-dimensional torus. As discussed in \cite{simasvalentim2} for the case $V(x)=x$, the intersection
\begin{equation*}
H^{\infty}_{W,V}(\mathbb{T}^d) := \bigcap_{n=1}^{\infty} H^{n}_{W,V}(\mathbb{T}^d)
\end{equation*}
forms a Fréchet space in which the regularity of the equation $L_{W,V} u = f$ is preserved; that is, $f \in H^{\infty}_{W,V}(\mathbb{T}^d)$ implies $u \in H^{\infty}_{W,V}(\mathbb{T}^d)$. In the same work, the authors modified the norms of $H^{n}_{W,V}$ to prove that $H^{\infty}_{W,V}$ becomes nuclear under a specific topology.

Our second objective in this work is to show that nuclearity holds for $H^{\infty}_{W,V}(\mathbb{T})$ with respect to the original norms of the high-order Sobolev spaces, without requiring the modifications introduced in \cite{simasvalentim2}. Furthermore, in the one-dimensional case, we have the identity $H^{\infty}_{W,V}(\mathbb{T}) = C_{W,V}^{\infty}(\mathbb{T})$. In short, our strategy to show that $C_{W,V}^{\infty}(\mathbb{T})$ is nuclear relies on establishing the existence of $0<\rho$ such that for every $s>\rho$,
\begin{equation}\label{s-series_conv}
    \sum_{n\in\mathbb{N}} \frac{1}{\lambda_{n}^{s}}<\infty,
\end{equation}
where $\{\lambda_n\}$ are the eigenvalues of the operator. For this purpose, it suffices to show that for some $s>0$, the series in \eqref{s-series_conv} converges. As we shall see, the relation \eqref{s-series_conv} arises as a consequence of the spectral asymptotics obtained for $\{\lambda_{n}\}_{n\in\mathbb{N}}$. More precisely, after characterizing the eigenvalues of \eqref{eigeneq} as roots of the entire function 
$$f(z)=\sum_{n\geq 1}(-1)^{n}z^{n}(F_{2n}(1,1)+G_{2n}(1,1))$$
we verify that for $0<\rho\le \frac{1}{2}$ the following asymptotic estimate holds:
\begin{equation}\label{assymtop_esti_proved}
    Cn^{\frac{1}{\rho}}\le \lambda_{n}.
\end{equation}
The upper bound for $\rho$, namely the value $\frac{1}{2}$ is sharp in the sense that it is achieved when $W(x)=V(x)=x$. It is important to emphasize that, as we will see, such sharp upper bound is a consequence of the following asymptotically sharp estimation:
\begin{equation}\label{tight_est}
	|F_{2n}(1,1)+G_{2n}(1,1)|\le\dfrac{C^{n}}{(n!)^2}.
\end{equation}

It is important to point out that we have not explicitly specified the value of $\rho$ in \eqref{assymtop_esti_proved}. However, as discussed in \cite{Freiberg_2005}, for the case where $W$ and $V$ are self-similar measures on $[0,1]$ (which differs from the main setting considered in this work), it is possible to explicitly determine a value of $\rho$ that is strictly less than $\frac{1}{2}$ for which \eqref{assymtop_esti_proved} holds for the operator $\Delta_{W,V}$ with Dirichlet (or Neumann) boundary conditions. While the result in \cite{Freiberg_2005} relies on the self-similarity property, our proof remains valid for self-similar measures under Dirichlet (or Neumann) boundary conditions. This is essentially due to the fact that the eigenvalues of $\Delta_{W,V}$ can be characterized as the roots of an entire function, as established in \cite[Proposition 3.7]{trig} and our tight estimation given in \eqref{tight_est}.

We emphasize that our asymptotic results were obtained under the assumption that $(\Delta_{W,V})^{-1}$ is compact, without assuming any stronger properties. As a consequence, before establishing the convergence of \eqref{s-series_conv}, which holds in particular for $s > \frac{1}{2}$, there is no information available regarding the traceability of $(\Delta_{W,V})^{-1}$. To highlight this point and motivated by the approach taken in our proof, we present two abstract results: the first addresses the spectral asymptotics of compact operators whose eigenvalues are roots of entire functions, while the second concerns the improvement of traceability for trace-class operators.

Another reason why our asymptotic result in \eqref{assymtop_esti_proved} is nontrivial is related to other findings also presented in this work. As we shall see, when considering the generalized Laplacian with Dirichlet boundary conditions, $\Delta_{W,V,\mathcal{D}}$, it is not necessary to use \eqref{assymtop_esti_proved} to find some $s>0$ such that \eqref{s-series_conv} holds and, consequently, to show that $C^{\infty}_{W,V,\mathcal{D}}(\mathbb{T}):=\bigcap_{n=1}^{\infty} H^{n}_{W,V,\mathcal{D}}(\mathbb{T})\subset C^\infty_{W,V}(\mathbb{T})$ is a nuclear space. More precisely, we show that $(-\Delta_{W,V,\mathcal{D}})^{-1}$ is a trace-class operator in $L^{2}_{V}(\mathbb{T})$ with kernel
\begin{equation}\label{kernel_dirich}
   \varrho_{W,0}(s,t)=W(t\wedge s) - \frac{W(t)W(s)}{W(1)}.
\end{equation}
As a consequence, in this case, the relation \eqref{s-series_conv} holds for the eigenvalues of $-\Delta_{W,V,\mathcal{D}}$ when $s=1$. However, this is not sufficient to conclude that the eigenvalues of $\Delta_{W,V}$ satisfy \eqref{s-series_conv} with $s=1$. Indeed, by the Courant min-max principle, if $\lambda_{i,\mathcal{D}}$ denotes the eigenvalues of $\Delta_{W,V,\mathcal{D}}$, we only have $\lambda_{i}\le \lambda_{i,\mathcal{D}}$. In particular, we cannot use the trace-class property in the Dirichlet case to deduce the nuclearity of $C^\infty_{W,V}(\mathbb{T})$. This demonstrates that the relations \eqref{s-series_conv} and \eqref{assymtop_esti_proved} are relevant. Moreover, as a consequence of $\lambda_{i}\le \lambda_{i,\mathcal{D}}$ and \eqref{assymtop_esti_proved}, we have $Cn^{\frac{1}{\rho}}<\lambda_{i,\mathcal{D}}$, which shows that $(-\Delta_{W,V,\mathcal{D}})^{-1}$ and its $s$-fractional power $(-\Delta_{W,V,\mathcal{D}})^{-s}$, defined via the classical Borel functional calculus for every $s>\rho$, are trace-class operators. This fact will be important in our applications section, where we determine the range of parameters for which generalized Whittle–Matérn type equations have well-defined solutions in $L^{2}(\Omega;L^{2}_{V}(\mathbb{T}))$ and also establish the regularity of their sample paths.

\subsection{Applications}
After introducing the higher-order and fractional $W$-$V$-Sobolev spaces $H^{s}_{W,V}(\mathbb{T})$ for $s>0$, we prove Sobolev-type embeddings as well as density of the space $C^{\infty}_{W,V}(\mathbb{T})$ in these spaces. Using our asymptotic results, we establish the nuclearity of $C^{\infty}_{W,V}(\mathbb{T})$. We then present a regularity result for the equation
\begin{equation}\label{eq:reg_whittle_matern}
	(I-\Delta_{W,V})^{s}u = f,
\end{equation}
showing in particular that, if the source term $f$ belongs to $C^{\infty}_{W,V}(\mathbb{T})$, then the solution $u$ lies in $C^{\infty}_{W,V}(\mathbb{T})$.

Additionally, we provide a regularity result for the stochastic Whittle–Matérn-type equation
\begin{equation}\label{eq:reg_whittle_matern_stoch}
	(\kappa^2 I - \Delta_{W,V})^\beta u = \dot{B}_V,
\end{equation}
where $\dot{B}_V$ denotes Gaussian white noise on $L^{2}_{V}(\mathbb{T})$ and $\kappa$ is a bounded and bounded away from zero function on $\mathbb{T}$. Specifically, we show that \eqref{eq:reg_whittle_matern_stoch} is well-posed as a process with sample paths in $L^{2}_{V}(\mathbb{T})$ if and only if $2\beta > \rho$. This highlights the significance of the upper bound for $\rho$ established in \eqref{assymtop_esti_proved}.

Finally, and importantly, we apply the results of Perez-Abreu and Kallianpur \cite{kallianpur_1987} to state a result on the existence and uniqueness of solutions to time-dependent stochastic partial differential equations valued in the dual of a nuclear space. More precisely, we discuss the existence and uniqueness of solutions to the equation
$$dY_t = \alpha \Delta_{W,V}' Y_t \, dt + \beta \, dN_t,$$
where $N_t$ is a mean-zero martingale taking values in the new space of distributions $D^{\prime}_{W,V}(\mathbb{T}) := \left(C^{\infty}_{W,V}(\mathbb{T})\right)^{\prime}$, and $\alpha$ and $\beta$ are positive constants. At this point, the nuclearity of $C^{\infty}_{W,V}(\mathbb{T})$ plays a crucial role in the discussion.

\subsection{Organization of the Manuscript}

The remainder of the paper is organized as follows. After this introduction, Section~\ref{sec:review} provides a review of $W$-$V$-Sobolev spaces and the generalized Laplacian. Section~\ref{sec:maclaurin} develops the generalized Maclaurin expansion and explores its consequences for the spectral theory of the $W$-$V$-Laplacian. Section~\ref{sect42} characterizes the eigenvectors of the $W$-$V$-Laplacian using generalized trigonometric functions. Section~\ref{sec:trace} presents preliminary results and investigates the traceability of the generalized Laplacian under Dirichlet boundary conditions. Section~\ref{sec:asymptotic_eigenvalues} establishes sharp asymptotic lower bounds for the eigenvalues. Section~\ref{sec:applications} discusses several applications of the theoretical results, including higher-order $W$-$V$-Sobolev spaces, the nuclearity of the associated function spaces, regularity of solutions of fractional stochastic differential equations, and stochastic partial differential equations in this generalized framework. Finally, Appendix \ref{app:spde_kallianpur} contains a brief review of parabolic stochastic partial differential equations acting on nuclear spaces.

\section{Review of $W$-$V$-Sobolev Spaces and the Generalized Laplacian}\label{sec:review}

In this section, we briefly recall some definitions and results from \cite{keale} regarding $W$-$V$-Sobolev spaces. First, a function is called c\`adl\`ag (from the French, ``continue à droite, limite à gauche'') if it is right-continuous and has left limits. Similarly, a function is called c\`agl\`ad if it is left-continuous and has right limits.

Fix two strictly increasing functions $W, V : \mathbb{T} \to \mathbb{T}$, with $W$ being c\`adl\`ag and $V$ being c\`agl\`ad. We also assume they satisfy the following periodic conditions:
\begin{equation}\label{eqsmedida}
	\forall x \in \mathbb{R},\quad 
	\begin{cases}
		W(x+1)-W(x)=W(1)-W(0);\\
		V(x+1)-V(x)=V(1)-V(0), 
	\end{cases}
\end{equation}
where, without loss of generality, we assume that both $W$ and $V$ are continuous at zero, i.e., $W(0-)=W(0)=0$ and $V(0+)=V(0)=0$. Indeed, since they are increasing, they can only have countably many points of discontinuity. If they are not continuous at zero, we can simply choose another point where both are continuous and translate the functions to make them continuous at zero. Note that by taking any compact interval $I$ in $\mathbb{R}$ of length $1$, we can define finite measures $dV$ and $dW$ on $\mathbb{T}$ by $dW((a,b])=W(b)-W(a)$ and $dV([a,b))=V(b)-V(a)$ for $a, b \in I$, $a < b$. Thus, $dW$ and $dV$ are uniquely determined by their values on the families $\{(a,b]: a < b, a, b \in I\}$ and $\{[a,b): a < b, a, b \in I\}$, respectively, and both these classes generate the Borel $\sigma$-algebra of $\mathbb{T}_I \cong \mathbb{T}$, where $\mathbb{T}_I$ is the torus obtained by identifying the boundary points of $I$.

Note that in the above definition, we allow the measures induced by $W$ and $V$ to have atoms. This is a weaker assumption than some commonly found in the literature, e.g., \cite{trig}, \cite{uta}, among others.

We denote the $L^2$ space with the measure induced by $V$ by $L^2_V(\mathbb{T})$, with its norm (resp. inner product) denoted by $\|\cdot\|_V$ (resp. $\langle \cdot, \cdot \rangle_V$). Similarly, the $L^2$ space with the measure induced by $W$ is denoted by $L^2_W(\mathbb{T})$, with norm (resp. inner product) $\|\cdot\|_W$ (resp. $\langle \cdot, \cdot \rangle_W$). The subspace of $L^2_V(\mathbb{T})$ (resp. $L^2_W(\mathbb{T})$) consisting of functions $f$ such that $\int_{\mathbb{T}} f dV = 0$ (resp. $\int_{\mathbb{T}} f dW = 0$) is denoted by $L^2_{V,0}(\mathbb{T})$ (resp. $L^2_{W,0}(\mathbb{T})$).
	
\begin{definition}\label{defgendif2}
A function $f : \mathbb{T} \to \mathbb{R}$ is said to be $W$-left differentiable if, for every $x \in \mathbb{T}$, the limit
$$D^{-}_{W}f(x):=\lim_{h\to 0^{-}}\frac{f(x+h)-f(x)}{W(x+h)-W(x)}$$
exists. Similarly, a function $g : \mathbb{T} \to \mathbb{R}$ is $V$-right differentiable if the limit
$$D^{+}_{V}g(x):=\lim_{h\to 0^{+}}\frac{g(x+h)-g(x)}{V(x+h)-V(x)}$$
exists for all $x \in \mathbb{T}$.
\end{definition}

Define the sets
$$C_{0}(\mathbb{T}):=\left\{f:\mathbb{T}\to\mathbb{R} : f\ \text{is c\`adl\`ag},\ \int_{\mathbb{T}} f dV = 0\right\},$$
$$C_{1}(\mathbb{T}):=\left\{f:\mathbb{T}\to\mathbb{R} : f\ \text{is c\`agl\`ad},\ \int_{\mathbb{T}} f dW = 0\right\},$$
and
$$C^{n}_{W,V,0}(\mathbb{T}):=\left\{f\in C_{\sigma(n+1)}(\mathbb{T}) : D^{(n)}_{W,V}f\ \text{exists and}\ D^{(n)}_{W,V}f\in C_{\sigma(n)}(\mathbb{T})\right\},$$
where
\begin{equation}\label{nderiv}
D^{(n)}_{W,V}:=\left\{
\begin{array}{ll}
\underbrace{D^{-}_{W}D^{+}_{V}\cdots D^{-}_{W}}_{n\ \text{factors}}, & \text{if } n \text{ is odd}; \\
\underbrace{D^{-}_{V}D^{+}_{W}\cdots D^{-}_{W}}_{n\ \text{factors}}, & \text{if } n \text{ is even},
\end{array}
\right.
\end{equation}
and
\begin{equation*}
\sigma(n):=\left\{
\begin{array}{ll}
1, & \text{if } n \text{ is odd}; \\
0, & \text{if } n \text{ is even}.
\end{array}
\right.
\end{equation*}

The space of $W$-$V$-smooth functions, or test functions, is defined as
$$C^{\infty}_{W,V}(\mathbb{T}):=\langle 1 \rangle \oplus \left(\bigcap_{n=1}^{\infty} C^{n}_{W,V,0}\right).$$

These functions allow us to define weak derivatives:

\begin{definition}\label{defWleft}
A function $f \in L^{2}_{V}(\mathbb{T})$ has a $W$-left \textit{weak} derivative if and only if, for every $g \in C^{\infty}_{V,W}(\mathbb{T})$, there exists $F \in L^{2}_{W}(\mathbb{T})$ such that
\begin{equation}\label{defWd2}
	\int_{\mathbb{T}} f D_{V}^{+}g\, dV = -\int_{\mathbb{T}} F g\, dW.
\end{equation}
In this case, the $W$-left weak derivative of $f$ is denoted by $D_W^- f$. We use the same notation as for the strong lateral derivative, in view of Remark 9 of \cite{keale}.
\end{definition}

We can now combine Definition 4 and Theorem 5 from \cite{keale} to arrive at the following definition of $W$-$V$-Sobolev spaces.

\begin{definition}
The $W$-$V$-Sobolev space is the Hilbert space
$${H}_{W,V}(\mathbb{T}) = \{f \in L^{2}_{V}(\mathbb{T}) : f\ \text{has a $W$-left weak derivative}\},$$
with norm
$$\|f\|^{2}_{W,V} = \|f\|^{2}_{V} + \|D_{W}^{-}f\|^{2}_{W}.$$
\end{definition}

It is noteworthy that our set of smooth functions, $C^{\infty}_{V,W}(\mathbb{T})$, is dense in $H_{W,V}(\mathbb{T})$. Indeed, the following is \cite[Proposition 2]{keale}:
\begin{proposition}\label{densenesscinfinity2}
The space $C^{\infty}_{V,W}(\mathbb{T})$ is dense in $L^2_V(\mathbb{T})$. Furthermore, $\left\{D^{+}_{V}g : g \in C^{\infty}_{V,W}(\mathbb{T})\right\}$ is dense in $L^2_{V,0}(\mathbb{T})$.
\end{proposition}

We also have a characterization of the $W$-$V$-Sobolev spaces that can be seen as a counterpart to the result that any function in $H^1(\mathbb{T})$ is absolutely continuous. More precisely, the following is a direct consequence of \cite[Theorem 2]{keale}.
\begin{theorem}\label{sobchar2}
A function $f \in L^{2}_{V}(\mathbb{T})$ belongs to $H_{W,V}(\mathbb{T})$ if and only if there exists $F \in L^{2}_{W,0}(\mathbb{T})$ such that
$$
	f(x) = f(0) + \int_{(0,x]} D_W^- F(s)\, dW(s),
$$
for $V$-almost every $x$.
\end{theorem}

Let us now introduce a generalization of the Laplacian. To this end, we first introduce a new space of functions, on which the Laplacian will be a self-adjoint operator.

\begin{definition}\label{defdomlap}
Let $\mathcal{D}_{W,V}(\mathbb{T})$ be the set of functions $f \in L^{2}_{W}(\mathbb{T})$ such that there exists $\mathfrak{f} \in L^{2}_{V,0}(\mathbb{T})$ satisfying
\begin{equation}\label{c422}
	f(x) = a + W(x) b + \int_{(0,x]} \int_{[0,y)} \mathfrak{f}(s)\, dV(s)\, dW(y),
\end{equation}
where $b$ satisfies the relation
$$b W(1) + \int_{(0,1]} \int_{[0,y)} \mathfrak{f}(s)\, dV(s)\, dW(y) = 0.$$
\end{definition}

We have the following integration by parts formula, whose proof can be found in \cite[Proposition 1]{keale}, involving functions in $\mathcal{D}_{W,V}(\mathbb{T})$ and ${H}_{W,V}(\mathbb{T})$.

\begin{proposition}\label{intbypartsprop2}
(Integration by parts formula) For every $f \in \mathcal{D}_{W,V}(\mathbb{T})$ and $g \in H_{W,V}(\mathbb{T})$, the following holds:
\begin{equation}\label{intbypartsVW2}
	\langle -\Delta_{W,V} f, g \rangle_{V} = \int_{\mathbb{T}} D^{-}_{W} f\, D^{-}_{W} g\, dW
\end{equation}
\end{proposition}

We can now define the $W$-$V$-Laplacian:

\begin{definition}\label{laplacianWV2}
We define the $W$-$V$-Laplacian as $\Delta_{W,V}: \mathcal{D}_{W,V}(\mathbb{T}) \subseteq L^{2}_{V}(\mathbb{T}) \to L^{2}_{V}(\mathbb{T})$, given by $\Delta_{W,V} f = \mathfrak{f}$, where $\mathfrak{f}$ is defined by \eqref{c422}.
\end{definition}

From \cite[Definition 5]{keale}, we have that $\mathcal{D}_{W,V}(\mathbb{T})$ is the domain of the Friedrichs extension of $I - D^{+}_{V} D^{-}_{W}$, denoted by $\mathcal{A}$ (see Zeidler \cite[Section 5.5]{zeidler} for further details on Friedrichs extensions). Therefore, the formal $W$-$V$-Laplacian, $\Delta_{W,V}: \mathcal{D}_{W,V}(\mathbb{T}) \subseteq L^{2}_{V}(\mathbb{T}) \to L^{2}_{V}(\mathbb{T})$, is self-adjoint. Furthermore, by \cite[Theorem 3]{keale}, $(I - \Delta_{W,V})^{-1}$ is well-defined and compact. Therefore, there exists a complete orthonormal system of functions $(\nu_n)_{n \in \mathbb{N}}$ in $L^{2}_{V}(\mathbb{T})$ such that $\nu_n \in H_{W,V}(\mathbb{T})$ for all $n$, and $\nu_n$ solves the equation $(I - \Delta_{W,V}) \nu_n = \gamma_n \nu_n$ for some $\{\gamma_n\}_{n \in \mathbb{N}} \subset \mathbb{R}$. Furthermore,
$$1 \le \gamma_1 \le \gamma_2 \le \cdots \to \infty$$
as $n \to \infty$. Now, observe that $\nu \in L^2_V(\mathbb{T})$ is an eigenvector of $\Delta_{W,V}$ with eigenvalue $\lambda$ if and only if $\nu$ is an eigenvector of $\mathcal{A}$ with eigenvalue $\gamma = 1 + \lambda$. Thus, $\{\nu_n, \lambda_n\}_{n \in \mathbb{N}}$ forms a complete orthonormal system of $L^{2}_{V}(\mathbb{T})$, where $\lambda_n = \gamma_n - 1$. Moreover,
\begin{equation}\label{asymplambda}
0 = \lambda_0 \le \lambda_1 \le \lambda_2 \le \cdots \to \infty,
\end{equation}
as $n \to \infty$, where for each $k \in \mathbb{N} \setminus \{0\}$, $\lambda_k$ satisfies
\begin{equation}\label{autovt2}
-\Delta_{W,V} \nu_k = \lambda_k \nu_k.
\end{equation}

This allows us to describe the elements of $H_{W,V}(\mathbb{T})$ in terms of their Fourier coefficients. Indeed, \cite[Theorem 6]{keale} tells us that
\begin{equation}\label{fourierHWV}
	H_{W,V}(\mathbb{T}) = \left\{f \in L^{2}_{V}(\mathbb{T}) : f = \alpha_0 + \sum_{i=1}^{\infty} \alpha_i \nu_i,\ \sum_{i=1}^{\infty} \lambda_i \alpha_i^2 < \infty \right\}.
\end{equation}

Finally, it is important to note that we can interchange the roles of $V$ and $W$ and obtain ``dual'' versions of the above results. More precisely,

\begin{remark}\label{orderVW2}
	In the rest of the paper, we need to pay attention to the position of $W$ and $V$ in the subscript. We will use the subscript $W,V$ for every structure strictly related to the operator $D^{+}_{V} D^{-}_{W}$. Similarly, whenever we use the subscript $V,W$, we will be referring to the analogous structure related to the operator $D^{-}_{W} D^{+}_{V}$. Notice that by doing this, we will keep changing between c\`adl\`ag and c\`agl\`ad functions. We ask the reader to be attentive to these details, as they are subtle.
\end{remark}

Below, we provide a combination of Lemmas 3 and 4 from \cite{keale}. 
\begin{lemma}\label{Wderivclosed}
The space
$$\mathcal{W} = \{D_W^- f : f \in H_{W,V}(\mathbb{T})\}$$
is a closed subspace of $L^{2}_{W,0}(\mathbb{T})$. Moreover,
the set $\left\{\frac{1}{\sqrt{\lambda_{i}}} D^{-}_{W} \nu_{i}\right\}_{i=1}^{\infty}$, where $\nu_{k}$ satisfies (\ref{autovt2}), is a complete orthonormal set in $\mathcal{W}$.
\end{lemma}

The following theorem is a version of the fundamental theorem of calculus:

\begin{theorem}
Let $W$ be a strictly increasing function on $\left[c_1, c_2\right]$. Assume that $f$ is left-continuous, $W$-locally integrable, and that $F$ and $W$ are right-continuous on $\left[c_1, c_2\right)$. Then, the following statements are equivalent:
	\begin{enumerate}
		\item  $D_W^{-} F(x) = f(x)$ for all $x \in \left(c_1, c_2\right]$;
		\item $\int_{\left(a_1, a_2\right]} f(y)\, dW(y) = F(a_2) - F(a_1)$ whenever $c_1 \leq a_1 \leq a_2 \leq c_2$.
	\end{enumerate}
\end{theorem}

There is also the following version of the integration by parts formula with boundary terms:

\begin{theorem}\label{thm:intbyparts}
	Let $f$ and $g$ be well-defined functions on $[a, b]$. Assume that $g$ is a c\`adl\`ag and $W$-left differentiable function with $D_W^{-}g$ c\`agl\`ad, and that $f$ is a c\`agl\`ad, $V$-right differentiable function with $D_V^{+}f$ c\`adl\`ag. Then,
	$$\int_{[a, b)} g(s) D_V^{+} f(s)\, dV(s) = [f(b) g(b) - f(a) g(a)] - \int_{(a, b]} f(s) D_W^{-} g(s)\, dW(s).$$
\end{theorem}

\section{Generalized Maclaurin Expansion for Functions in $C^{\infty}_{W,V}(\mathbb{R})$}\label{sec:maclaurin}

The goal of this section is to provide an analogue of the Maclaurin series expansion for functions in $C^{\infty}_{W,V}(\mathbb{R})$. Since we are interested in the Maclaurin expansion, our focus is on investigating the behavior of functions in $C^{\infty}_{W,V}(\mathbb{R})$ near the origin. The same approach can be used to obtain similar representations around any other point in $\mathbb{R}$. However, as our main applications are for the torus, several simplifications are possible. Indeed, since we can rotate the torus and choose any continuity point of $W$ as the base point, we may assume without loss of generality that $W(0)=0$ (since a vertical translation of $W$ induces the same measure $dW$). This shows that the Maclaurin expansion is well-suited for the torus case. 

Because we are interested in local properties near the origin, instead of working with $C^{\infty}_{W,V}(\mathbb{R})$, we may consider $C^{\infty}_{W,V}([a,b])$ with $0\in(a,b)$. Since $\int_{(a,b]} f\,dW = -\int_{(b,a]} f\,dW$ when $b<a$, we can focus our attention on the case $C^{\infty}_{W,V}([0,1])$. 

Our first proposition is a simple consequence of the integration by parts formula and provides insight into which functions play the role of the monomials $\frac{x^n}{n!}$ in this generalized framework.

\begin{proposition}
	For $0\le s\le x\le 1$, define $F_{1}(x,s)=W(s)$,
\begin{equation*}
	F_{n}(x,s)=\begin{cases}
		\displaystyle\int_{[0,s)}\left[F_{n-1}(x,x)-F_{n-1}(x,\xi)\right]dV(\xi),\quad \text{n even};\\
		\displaystyle\int_{(0,s]}\left[F_{n-1}(x,x)-F_{n-1}(x,\xi)\right]dW(\xi),\quad \text{n odd}.
	\end{cases}
\end{equation*}
Then, if $f\in C_{W,V}([0,1])$, there is a uniquely determined function $I_{n}:[0,1]\to\mathbb{R}$ such that
\begin{equation*}
	f(x)-f(0) = \sum_{k=1}^{n-1}\left[D^{(k)}_{W,V}f(0)\right]F_{k}(x,x)+I_{n}(x)
\end{equation*}
for every integer $n \geq 2$ and $x\in [0,1]$. 
\end{proposition}

\begin{proof}
Let $f\in C^{\infty}_{W,V}([0,1])$. From the integration by parts formula, using that $D^{-}_{W}W=1$, we obtain
\begin{align*}
	f(x)-f(0) &= \int_{(0,x]}D^{-}_{W}f(s)dW(s) = \int_{(0,x]}D^{-}_{W}f(s)D^{-}_{W}W\, dW(s) \\
	&= D^{-}_{W}f(x)W(x)-\int_{[0,x)}D^{+}_{V}D^{-}_{W}f(s)W(s)dV(s) \\
	&= D^{-}_{W}f(0)W(x)+\int_{[0,x)}D^{+}_{V}D^{-}_{W}f(s)\left[W(x)-W(s)\right]dV(s) \\
	&= D^{-}_{W}f(0)W(x)+I_{2}(x),
\end{align*}
where $I_{2}(x)=\int_{[0,x)}D^{+}_{V}D^{-}_{W}f(s)\left[W(x)-W(s)\right]dV(s)$. Now, define 
$$F_{2}(x,s)=V(s)W(x)-\int_{[0,s)}W(\xi)dV(\xi)=\int_{[0,s)}\left[W(x)-W(\xi)\right]dV(\xi).$$ 
We have that $D^{+}_{V,s}F_{2}(x,s)=W(x)-W(s)$. We can use integration by parts again, this time interchanging the roles of $V$ and $W$ (see Remark \ref{orderVW2}), to obtain
\begin{align*}
	I_{2}(x) &= \int_{[0,x)}D^{+}_{V}D^{-}_{W}f(s)\left[W(x)-W(s)\right]dV(s) \\
	&= D^{+}_{V}D^{-}_{W}f(x)F_{2}(x,x)-D_{V}^{+}D_{W}^{-}f(0)F_{2}(x,0)-\int_{(0,x]}D^{-}_{W}D^{+}_{V}D^{-}_{W}f(s)F_{2}(x,s)dW(s) \\
	&= D^{+}_{V}D^{-}_{W}f(x)F_{2}(x,x)-\int_{(0,x]}D^{-}_{W}D^{+}_{V}D^{-}_{W}f(s)F_{2}(x,s)dW(s)\\
	&= D^{+}_{V}D^{-}_{W}f(0)F_{2}(x,x)+\int_{(0,x]}D^{-}_{W}D^{+}_{V}D^{-}_{W}f(s)\left[F_{2}(x,x)-F_{2}(x,s)\right]dW(s)\\
	&= D^{+}_{V}D^{-}_{W}f(0)F_{2}(x,x)+I_{3}(x).
\end{align*}
where $I_{3}(x)=\int_{(0,x]}D^{-}_{W}D^{+}_{V}D^{-}_{W}f(s)\left[F_{2}(x,x)-F_{2}(x,s)\right]dW(s).$
In summary,
$$
f(x)-f(0) = D^{-}_{W}f(0)W(x)+D^{+}_{V}D^{-}_{W}f(0)F_{2}(x,x)+I_{3}(x).
$$
Inductively, for all $n \geq 2$ we have
\begin{align}\label{taylorpolynomial}
	f(x)-f(0) = \sum_{k=1}^{n-1}\left[D^{(k)}_{W,V}f(0)\right]F_{k}(x,x)+I_{n}(x),
\end{align}
where $D^{(k)}_{W,V}$ is the operator defined in \eqref{nderiv}, $F_{1}(x,s)=W(s)$,
\begin{equation}\label{term:1}
	F_{n}(x,s)=\begin{cases}
		\displaystyle\int_{[0,s)}\left[F_{n-1}(x,x)-F_{n-1}(x,\xi)\right]dV(\xi),\quad \text{n even};\\
		\displaystyle\int_{(0,s]}\left[F_{n-1}(x,x)-F_{n-1}(x,\xi)\right]dW(\xi),\quad \text{n odd},
	\end{cases}
\end{equation}
and the remainder is given by
\begin{equation*}
	I_n(x)=\begin{cases}
		\displaystyle\int_{[0, x)} D_{W, V}^{(n+1)} f(s)\left[F_n(x, x)-F_n(x, s)\right] d W(s), & n \text{ even}; \\ [2ex]
		\displaystyle\int_{(0, x]} D_{W, V}^{(n+1)} f(s)\left[F_n(x, x)-F_n(x, s)\right] d V(s), & n \text{ odd.}
	\end{cases}
\end{equation*}
\end{proof}

The next lemma will help in estimating the remainder $I_{n}(x)$. In particular, it provides estimates for $F_{2n}(x,x)$ and $F_{2n+1}(x,x)$.
\begin{lemma}\label{taylorbound} 
If $(F_{n})_{n\in\mathbb{N}}$ is the sequence defined in \eqref{term:1}, then for $0 \le s\le x\le1$,
\begin{equation}\label{estimativadoresto}
	F_{2n}(x)-F_{2n}(s)\le \frac{\left(F_{2}(x)-F_{2}(s)\right)^n}{n!},
\end{equation}
where $F_{n}(s):=F_{n}(x,s)$. In particular, for some constant $C>0$ independent of $n$, we have
$$F_{2n+1}(x)\le C\frac{\left(F_{2}(x)\right)^n}{n!}.$$
\end{lemma}
\begin{proof}
	Assume that $s<x$. It is easy to see from \eqref{term:1} that $F_{n}(s)$ is increasing for $s\le x$. So,
	\begin{equation}\label{recorrencia1}
		\begin{split}
		F_{2n}(x)-F_{2n}(s) &=\int_{[s,x)}F_{2n-1}(x)-F_{2n-1}(\xi)dV(\xi)\\ 
		&= \int_{[s,x)}\left(\int_{(\xi,x]}F_{2(n-1)}(x)-F_{2(n-1)}(\eta)dW(\eta) \right) dV(\xi) 
		\\ 
		&\le  \int_{[s,x)}\left(F_{2(n-1)}(x)-F_{2(n-1)}(\xi+)\right)(W(x)-W(\xi))dV(\xi) \\
		&= \int_{[s,x)}F_{2(n-1)}(x)-F_{2(n-1)}(\xi+)dF_{2}(\xi).
		\end{split}
	\end{equation}
	Similarly,
	\begin{equation}\label{recorrencia2}
		F_{2n}(x)-F_{2n}(\xi+)\le \int_{(\xi,x)} F_{2(n-1)}(x)-F_{2(n-1)}(\beta+)dF_{2}(\beta).   
	\end{equation}
	From (\ref{recorrencia1}) and (\ref{recorrencia2}), it is enough to obtain a bound for
	\begin{equation}\label{termointegral}
		\int_{[s,x)}dF_{2}(\xi_{1})\int_{(\xi_{1},x)}dF_{2}(\xi_{2})\cdots \int_{(\xi_{n-1},x)}(F_{2}(x)-F_{2}(\xi_{n}+))dF_{2}(\xi_{n}).
	\end{equation}
	Now, put $c:=F_{2}(x)$ and notice that
	$$D^{+}_{F_{2}}(c-F_{2})^{n+1}(\xi)=-\sum_{j=0}^{n}\left(c-F_{2}(\xi+)\right)^{j}\left(c-F_{2}(\xi)\right)^{n-j}\le -(n+1)(c-F_{2}(\xi+))^{n}.$$
	Integrating the above inequality and using that $\frac{(c-F_{2})^{n+1}}{n+1}$ is $F_{2}$-absolutely continuous in the sense of \cite{pouso}, we obtain
	\begin{equation}\label{desig1}
		\frac{(c-F_{2}(s))^{n+1}}{n+1}\geq \int_{[s,x)}(c-F_{2}(\xi+))^{n}dF_{2}(\xi).
	\end{equation}
	Now, given any $s^\ast <x$, we can take the limit from the right as $s\to s^\ast+$ in (\ref{desig1}) to obtain
	\begin{equation}\label{desig2}
		\frac{(c-F_{2}(s^\ast+))^{n+1}}{n+1}\geq \int_{(s^\ast,x)}(c-F_{2}(\xi+))^{n}dF_{2}(\xi).
	\end{equation}
	The inequality (\ref{estimativadoresto}) follows from successive applications of (\ref{desig1}) and (\ref{desig2}) to (\ref{termointegral}). The last consequence in the statement follows directly from the previous inequality, since
	\begin{align*}
		F_{2n}(x) &=\int_{(0,x]}F_{2n+1}(x)-F_{2n+1}(\xi)dV(\xi) \\
		&= \int_{(0,x]}\left(\int_{[\xi,x)}F_{2n}(x)-F_{2n}(\eta)dV(\eta) \right) dW(\xi) \\
		& \le \int_{(0,x]}\left(\int_{[\xi,x)}\frac{\left(F_{2}(x)-F_{2}(\eta)\right)^n}{n!} dV(\eta) \right) dW(\xi) \\
		& \le \frac{\left(F_{2}(x)\right)^n}{n!} \int_{(0,x]}\left(\int_{[\xi,x)} dV(\eta) \right) dW(\xi) \le \frac{\left(F_{2}(x)\right)^n}{n!} W(1)V(1).
	\end{align*}
\end{proof}

The next lemma can be viewed as a version of the Leibniz integral rule with respect to the $W$-left-derivative. It will help us when computing certain $W$-left-derivatives related to $F_{n}$.
\begin{lemma}\label{leibnizrule}
	Let $g:[0,1]\times [0,1]\to \mathbb{R}$ satisfy:
	\begin{enumerate}
		\item For all $x\in [0,1]$, $g(x,\cdot):[0,1]\to \mathbb{R}$ is a càglàd function;
		\item There exists $D^{-}_{W,1}g:(0,1]\times [0,1]\to \mathbb{R}$ such that $D^{-}_{W,1}g(x,\cdot):[0,1]\to \mathbb{R}$ is càglàd and 
		$$\lim_{h\to 0} \sup_{\xi\in [0,1]}\left|\frac{g(x,\xi)-g(x-h,\xi)}{W(x)-W(x-h)}-D^{-}_{W,1}g(x,\xi)\right|=0;$$
		\item For all $x\in [0,1]$, $g(x,x)=D_{W,1}g(x,x)=0$.
	\end{enumerate}
	Then, 
	$$D^{-}_{W}\left(\int_{(0,\cdot]}g(\cdot,\xi)dW(\xi)\right)(x)=\int_{(0,x]}D^{-}_{W,1}g(x,\xi)dW(\xi)$$
	for all $x\in (0,1]$.
\end{lemma}
\begin{proof}
	Fix $x\in (0,1]$ and $h>0$ sufficiently small. Consider the quotient
	\begin{align*}
		\Delta(x,h) &=\frac{1}{W(x)-W(x-h)}\left(\int_{(0,x]}g(x,\xi)dW(\xi)-\int_{(0,x-h]}g(x-h,\xi)dW(\xi)\right)\\
		&= \frac{1}{W(x)-W(x-h)}\int_{(x-h,x]}g(x,\xi)dW(\xi)\\
		&\quad+\int_{(0,x-h]}\left[\frac{g(x,\xi)-g(x-h,\xi)}{W(x)-W(x-h)}-D_{W,1}^{-}g(x,\xi)\right]dW(\xi)\\
		&\quad+\int_{(0,x-h]}D_{W,1}^{-}g(x,\xi)dW(\xi) = E_{1}+E_{2}+E_{3}.
	\end{align*}
	From items 1 and 3, the term $E_{1}$ goes to zero. Moreover, from item 2, $E_{2}$ also vanishes. Again, from items 1 and 3, $E_{3}$ trivially goes to zero.
\end{proof}

The main point of Lemma \ref{leibnizrule} is the following result on the computation of the partial derivatives $D_{W,1}^{-}F_{n}(x,s)$ and $D^{-}_{W}F_{n}(x,x)$.

\begin{lemma}\label{taylorlemma}
	Let $G_{1}(x,s)=V(s)$ and recursively define
	\begin{equation}\label{Gnformula}
		G_{n}(x,s)=\begin{cases}
		\displaystyle\int_{(0,s]}\left[G_{n-1}(x,x)-G_{n-1}(x,\xi)\right]dW(\xi),\quad \text{n even};\\
		\displaystyle\int_{[0,s)}\left[G_{n-1}(x,x)-G_{n-1}(x,\xi)\right]dV(\xi),\quad \text{n odd}.
		\end{cases}
	\end{equation}
	Then,
\begin{equation*}
	\left\{
	\begin{array}{l}
		D^{-}_{W,1}F_{n}(\xi,s)=G_{n-1}(\xi,s),\\ [0.5em]
		D^{-}_{W}F_{n}(\xi,\xi)=G_{n-1}(\xi,\xi),
	\end{array}
	\right.
\end{equation*}
	where $D^{-}_{W,1}$ is as defined in item 2 of Lemma \ref{leibnizrule}.
\end{lemma}

\begin{proof}
	We proceed by induction on $n$. For $n=2$,
	\begin{equation}\label{eqsindepend}
		F_{2}(x,s)=\int_{[0,s)}[W(x)-W(\xi)]dV(\xi)=V(s)W(x)-\int_{[0,s)}W(\xi)dV(\xi),
	\end{equation}
	which implies $D^{-}_{W,1}F_{2}(x,s)=V(s)=G_{1}(x,s)$. Moreover,
	$\frac{F_{2}(x-h,s)-F_{2}(x,s)}{W(x-h)-W(x)}\to V(s)$ as $h\to0$, with this limit holding uniformly in $s$. From the integration by parts formula (see Proposition \ref{intbypartsprop2}), it follows that
	$$F_{2}(x,x)=\int_{(0,x]}V(\alpha)dW(\alpha).$$
	Therefore, $D^{-}_{W}F_{2}(x,x)=V(x)=G_{1}(x,x)$. 

	Assume that for an odd index $n$, for all $k\le n$ and all $x$, $D^{-}_{W,1}F_{k}(x,s)=G_{k-1}(x,s)$ exists uniformly in $s$, and $D^{-}_{W}F_{k}(x,x)=G_{k-1}(x,x)$. Now, for $F_{n+1}(x,x)$, we can use Fubini's theorem and the definition of $F_{n}$ to obtain
	\begin{align}\label{fubinada}
		\nonumber F_{n+1}(x,x) &=\int_{[0,x)}F_{n}(x,x)-F_{n}(x,\xi)dV(\xi)\\ \nonumber 
		&= \int_{[0,x)}\left(\int_{(\xi,x]}F_{n-1}(x,x)-F_{n-1}(x,\alpha)dW(\alpha)\right)dV(\xi)
		\\ &=\int_{(0,x]}[F_{n-1}(x,x)-F_{n-1}(x,\alpha)]V(\alpha)dW(\alpha).
	\end{align}
	By the induction hypothesis, the function $g(x,\alpha)=[F_{n-1}(x,x)-F_{n-1}(x,\alpha)]V(\alpha)$ satisfies all the conditions in Lemma \ref{leibnizrule}, with $D^{-}_{W,1}g(x,\alpha)=[G_{n-2}(x,x)-G_{n-2}(x,\alpha)]V(\alpha)$. So,
	$$D^{-}_{W}F_{n+1}(x,x)=\int_{(0,x]}[G_{n-2}(x,x)-G_{n-2}(x,\alpha)]V(\alpha)dW(\alpha).$$
	From the same computation in (\ref{fubinada}), we obtain $D^{-}_{W}F_{n+1}(x,x)=G_{n}(x,x)$. The uniform existence of $D^{-}_{W,1}F_{n+1}(x,s)$ follows easily from the same argument as in (\ref{eqsindepend}) and the definition of $G_{n}(x,s)$. For the case where $n$ is even, the result is straightforward because the measure we are integrating with respect to is $dW$, so all computations reduce to a direct application of Lemma \ref{leibnizrule} or the dominated convergence theorem.
\end{proof}

\begin{remark}\label{rmk2impor}
	The statements in Lemmas \ref{taylorbound}, \ref{leibnizrule}, and \ref{taylorlemma} also have versions relative to the $V$-right derivative and càglàd (càdlàg) functions, and the adaptations are straightforward from the ideas presented above. Thus, the $V$-right derivative of $G_{n}(x,s)$ or $G_{n}(x,x)$ with respect to $x$ also exists and equals $F_{n-1}(x,s)$ or $F_{n-1}(x,x)$, respectively.
\end{remark}

The following corollary is a straightforward consequence of Lemma \ref{taylorlemma} and Remark \ref{rmk2impor}. It essentially says that, if we set $F_{0}(x,x)\equiv 1$ and start with $dW$, the function $F_{n}(x,x)$ is obtained by alternating $n$ successive integrations, starting from $F_{0}$, from $0$ to $x$ with respect to $dW$ and $dV$.

\begin{corollary}\label{rmk111}
	Let $(p_{0}(x),q_{0}(x))=(1,W(x))$ and, for $n>0$, recursively define $(p_{n+1}(x),q_{n+1}(x))$ by
	$$(p_{n+1}(x),q_{n+1}(x))=\left(\int_{(0,x]}\int_{[0,s)}p_{n}(\xi)dV(\xi)dW(s),\int_{(0,x]}\int_{[0,s)}q_{n}(\xi)dV(\xi)dW(s)\right).$$
	If we assume $F_{0}(x,x)\equiv 1$, then $$(p_{n}(x),q_{n}(x))=(F_{2n}(x,x),F_{2n+1}(x,x))$$
	for all $n\in\mathbb{N}$. An analogous characterization for $\{G_{n}\}_{n\in\mathbb{N}}$ is immediate from Remark \ref{rmk2impor}.
\end{corollary}

The following theorem will be our main tool to provide nontrivial examples of functions in $C^{\infty}_{W,V}(\mathbb{T})$.

\begin{theorem}\label{analyticalVW}
	Let $f\in C^{\infty}_{W,V}([0,1])$ be a function such that 
	\begin{equation}\label{coefconducn}
		\max\left\{\|D^{(2n)}_{W,V}f\|_{\infty},\|D^{(2n+1)}_{W,V}f\|_{\infty}\right\}\le c_{n},
	\end{equation}
	where $\|\cdot\|_{\infty}$ denotes the sup-norm in $C^{\infty}_{W,V}([0,1])$ and $c_{n}=o\left(\frac{n!}{e^{n}}\right)$. Then, we have the following expansion for $f$:
	\begin{equation}\label{sdetaylor}
		f(x)=f(0)+\sum_{k=1}^{\infty}\left[D^{(k)}_{W,V}f(0)\right]F_{k}(x,x).
	\end{equation}
	The convergence in \eqref{sdetaylor} is uniform over $[0,1]$. 

	Moreover, in order for $f$ and $D^{(n)}_{W,V}f$ to be well defined on $\mathbb{T}$, the following set of equations must be satisfied:
	\begin{equation*}
	\left\{
	\begin{array}{l}
		\displaystyle\sum_{k=1}^{\infty} D^{(k)}_{W,V}f(0)\, F_{k}(1,1) = 0, \\ [0.5em]
		\displaystyle\sum_{k=1}^{\infty} D^{(k+1)}_{W,V}f(0)\, G_{k}(1,1) = 0, \\ [0.5em]
		\displaystyle\sum_{k=1}^{\infty} D^{(k+2)}_{W,V}f(0)\, F_{k}(1,1) = 0, \\
		\hspace{2em}\vdots
	\end{array}
	\right.
	\end{equation*}
\end{theorem}

\begin{proof}
	Let us estimate the remainder $I_{n}$ in \eqref{taylorpolynomial} in terms of $\|D^{n}f\|_{\infty}$. From Lemma \ref{taylorbound}, it follows that
	\begin{align*}
		|I_{2n}(x)|&=\left|\int_{(0,x]}D_{W,V}^{(2n+1)}f(s)\left[F_{2n}(x,x)-F_{2n}(x,s)\right]dW(s)\right|\le \frac{(F_{2}(1-,1-))^{n}}{n!}\|D_{W,V}^{(2n+1)}f\|_{\infty}
	\end{align*}
	and 
	\begin{align*}
		|I_{2n+1}(x)|&=\left|\int_{(0,x]}D_{W,V}^{(2n+2)}f(s)\left[F_{2n+1}(x,x)-F_{2n+1}(x,s)\right]dW(s)\right|\\
		&\le \|D_{W,V}^{(2n+2)}f\|_{\infty}\int_{[0,x)}\left[F_{2n+1}(x,x)-F_{2n+1}(x,s)\right]dV(s) \\
		&=\|D_{W,V}^{(2n+2)}f\|_{\infty}\int_{[0,x)}\left(\int_{(s,x]}F_{2n}(x,x)-F_{2n}(x,\alpha)dW(\alpha)\right)dV(s) \\
		&\le \|D_{W,V}^{(2n+2)}f\|_{\infty}\int_{[0,x)}F_{2n}(x,x)-F_{2n}(x,\alpha+) dF_{2}(s) \\
		&\le \|D_{W,V}^{(2n+2)}f\|_{\infty}\frac{F_{2}(1-,1-)^{n+1}}{(n+1)!}.
	\end{align*}
	The previous estimates for $I_{n}(x)$, together with (\ref{coefconducn}), yield the uniform convergence of 
	$$S_{n}(x):=f(0)+\sum_{k=1}^{n-1}\left[D^{(k)}_{W,V}f(0)\right]F_{k}(x,x)$$ 
	to $f(x)$. Finally, the conditions that $f$ and $D_{W,V}^{n}f$ must satisfy to be well-defined on the torus follow from Lemma \ref{taylorlemma} and the uniform convergence of $S_{n}(x)$.
\end{proof} 

The ideas of this section may be adapted to the domain $\mathbb{R}$, as long as we require our results to hold on each closed interval, not necessarily symmetric around the origin. However, such non-symmetric cases are not relevant for the applications we have in mind in this work.

% The expression \eqref{taylorpolynomial} can be seen as a Taylor's expansion (centered at zero) with integral remainder, whereas  Theorem \ref{analyticalVW}, provides a condition for a function $f\in C^{\infty}_{W,V}(\mathbb{T})$ to have a series representation given by \eqref{sdetaylor}. The functions in $C^{\infty}_{W,V}(\mathbb{T})$ that satisfy \eqref{sdetaylor} are our analogue of analytical functions. 

\section{Characterization of the eigenvectors of $-\Delta_{W,V}$}\label{sect42}

In this section, we use the result of Lemma \ref{taylorbound} to define generalized analogues of the usual trigonometric functions $\cos(\alpha x)$ and $\sin(\alpha x)$. These generalized trigonometric functions will be used to fully characterize the eigenvectors of $-\Delta_{W,V}$.

Let $\alpha>0$ and define the following functions from $\mathbb{R}$ to $\mathbb{R}$:
\begin{equation*}
C_{W,V}(\alpha,x)=\sum_{n=0}^{\infty}\alpha^{2n}(-1)^{n}F_{2n}(x,x),\qquad S_{W,V}(\alpha,x)=\sum_{n=0}^{\infty}\alpha^{2n+1}(-1)^{n}F_{2n+1}(x,x)
\end{equation*}
and
\begin{equation*}
C_{V,W}(\alpha,x)=\sum_{n=0}^{\infty}\alpha^{2n}(-1)^{n}G_{2n}(x,x),\qquad S_{V,W}(\alpha,x)=\sum_{n=0}^{\infty}\alpha^{2n+1}(-1)^{n}G_{2n+1}(x,x).
\end{equation*}

\begin{remark}
	The functions $C_{W,V}$ and $S_{W,V}$ (or $C_{V,W}$ and $S_{V,W}$) are well-defined. Indeed, by Lemma \ref{taylorbound}, each of these series converges uniformly and absolutely on every compact interval of $\mathbb{R}$.
\end{remark}

Notice that when $V(x)=W(x)=x$, the functions $C_{W,V}$ and $S_{W,V}$ (or $C_{V,W}$ and $S_{V,W}$) coincide with $\cos(\alpha x)$ and $\sin(\alpha x)$, respectively. Using Lemma \ref{leibnizrule} and the uniform convergence on compacts, it is straightforward to derive the following relations:
\begin{equation}\label{difcossinWV}
D^{-}_{W}C_{W,V}(\alpha,x)=-\alpha S_{V,W}(\alpha,x),\qquad D^{-}_{W}S_{W,V}(\alpha,x)=\alpha C_{V,W}(\alpha,x)
\end{equation}
and
\begin{equation}\label{difsincosVW}
D^{+}_{V}C_{V,W}(\alpha,x)=-\alpha S_{W,V}(\alpha,x),\qquad D^{+}_{V}S_{V,W}(\alpha,x)=\alpha C_{W,V}(\alpha,x)
\end{equation}
These follow directly. Furthermore, we can use the relations \eqref{difcossinWV} and \eqref{difsincosVW} to see that $S_{W,V}(\alpha,x)$ is the unique solution in $C^{\infty}_{W,V}(\mathbb{R})$ of
\begin{equation}\label{sineqpart}
	\left\{ \begin{array}{ll}
		-\Delta_{W,V}u=\alpha^2 u;\\
		u(0)=0,\, D^{-}_{W}u(0)=\alpha. \end{array} \right.
\end{equation}
and that $C_{W,V}(\alpha,x)$ is the unique solution in $C^{\infty}_{W,V}(\mathbb{R})$ of
\begin{equation}\label{coseqpart}
\left\{ \begin{array}{ll}
	-\Delta_{W,V}u=\alpha^2 u;\\
	u(0)=1,\, D^{-}_{W}u(0)=0. \end{array} \right.
\end{equation}

Now, observe that since $S_{W,V}(\alpha,\cdot)$ solves \eqref{sineqpart} and $C_{W,V}(\alpha,\cdot)$ solves \eqref{coseqpart}, it is easy to see that for any $\alpha\neq 0$, $S_{W,V}(\alpha,\cdot)$ and $C_{W,V}(\alpha,\cdot)$ are linearly independent. Consequently, any solution of $-\Delta_{W,V}u=\alpha^2 u$ is given by
\begin{equation*}
u(x)=A\, C_{W,V}(\alpha,x)+B\, S_{W,V}(\alpha,x)
\end{equation*}
where $A$ and $B$ are determined by the initial conditions at $x=0$.

Our next result characterizes the eigenvectors of $-\Delta_{W,V}$ on the torus using the functions $C_{W,V}(\alpha,\cdot)$ and $S_{W,V}(\alpha,\cdot)$. The reader is invited to compare this characterization with the Taylor expansions of the eigenvectors of the standard Laplacian, $-\Delta$, on the torus.

\begin{proposition}\label{caraceige}
	If $(\lambda_{i},\nu_{i})_{i>0}$ satisfy 
	(\ref{eigeneq}), then there exist $a_i,b_i\in\mathbb{R}$ such that
	\begin{equation}\label{s01}
		\nu_{i}(x)=a_{i}C_{W,V}\left(\sqrt{\lambda_{i}},x\right)+\dfrac{b_{i}}{\sqrt{\lambda_{i}}}S_{W,V}\left(\sqrt{\lambda_{i}},x\right),
	\end{equation}
	$V$-a.e. Furthermore, for each $i$, the vector $(a_{i},b_{i})\neq(0,0)$ is obtained as a solution of the system:
	\begin{equation}\label{bcondi}
		\begin{cases}
			a_{i}\left[C_{W,V}\left(\sqrt{\lambda_{i}},1\right)-1\right]+\dfrac{b_{i}}{\sqrt{\lambda_{i}}}S_{W,V}(\sqrt{\lambda_{i}},1)=0; \\
			\dfrac{a_{i}}{\sqrt{\lambda_{i}}}S_{V,W}(\sqrt{\lambda_{i}},1)-\dfrac{b_{i}}{\lambda_{i}}\left[C_{V,W}\left(\sqrt{\lambda_{i}},1\right)-1\right]=0.
		\end{cases}
	\end{equation}
\end{proposition}

\begin{proof}
	Let $(\nu_{i},\lambda_{i})$ be an eigenpair of $-\Delta_{W,V}$ with $\lambda_{i}>0$. As $\nu_{i}\in \mathcal{D}(-\Delta_{W,V})$, it follows from Definition \ref{defdomlap} that
	\begin{equation}\label{caracauto}
		\nu_{i}(x)=a_{i}+b_{i}W(x)-\lambda_{i}\int_{(0,x]}\int_{[0,s)}\nu_{i}(\xi)dV(\xi)dW(s),
	\end{equation}
	with $a_{i}$ and $b_{i}$ fulfilling the relations 
	\begin{equation}\label{bondcondaq}
		\begin{cases}
			\int_{\mathbb{T}}\nu_{i}dV=0;\\
			b_{i}W(1)+\int_{\mathbb{T}}\int_{[0,s)}\nu_{i}(\xi)dV(\xi)dW(s)=0.
		\end{cases}
	\end{equation}
	From Corollary \ref{rmk111} we have 
	\begin{equation}
F_{1}(x,x)=W(x), \quad F_{2}(x,x)= \int_{[0,x)} V(s)dW(s), \quad F_{3}(x,x)=\int_{(0,x]}\int_{(0,s]} W(\xi) dV(\xi)dW(s), \ldots
	\end{equation}
	Therefore, if we recursively substitute (\ref{caracauto}) into itself, it is easy to see from the previous characterization of $F_{n}$ that, up to a set of measure zero, we may write $\nu_{i}=S_{n}+R_{n}$,
	where $S_{n}$ has the form $S_{n}(x):=\nu_{i}(0)+\sum_{k=1}^{n-1}c_{k}F_{k}(x,x)$ and $R_{n}$ is the result of $n$ iterations of the integral with respect to the product measure $dV\otimes dW(d\xi,ds)$ on the function $1_{(0,s]}(\xi)\nu_{i}(\xi)$ over the interval $(0,x]\times [0,x)$. Since $\nu_{i}$ is bounded on $\mathbb{T}$ (because it is a well-defined càdlàg function on the torus), we may estimate $R_{n}$ similarly to how we estimate $I_{n}$ in Theorem \ref{analyticalVW} to obtain the uniform convergence of $S_{n}$ to $\nu_{i}$. Indeed, it is straightforward to compute $D^{(n)}_{W,V}\nu_{i}(0)$ and determine the coefficients $c_{k}$ in the analytic representation of $\nu_{i}$. Moreover, taking $s=0$ in Lemma \ref{taylorbound}, and using that $F_{2n+1}$ is determined from $F_{2n}$, it follows from the Weierstrass M-test that $S_{n}$ also converges absolutely. By absolute convergence, we may split the series representation of $\nu_{i}$ into the sum of the series containing the even terms and the series containing the odd terms. The expressions obtained after such decomposition yield the representation (\ref{caracauto}). Using the same iterative procedure, we may rewrite the conditions (\ref{bondcondaq}) according to (\ref{bcondi}).
\end{proof}

As a direct consequence of Proposition \ref{caracauto}, we obtain a strong version of \cite[Theorem 4]{keale} regarding the regularity of eigenvectors of $-\Delta_{W,V}$.

\begin{corollary}[Regularity of the eigenvectors]\label{regeigenvect}
	If $u\in L^{2}_{V}(\mathbb{T})\setminus \{0\}$ satisfies
	\begin{equation*}
		\begin{cases}
			\Delta_{W,V}u=\lambda u;\\
			u\in \mathcal{D}_{W,V}(\mathbb{T}).
		\end{cases}
	\end{equation*}
	for some $\lambda>0$, then there exists $v\in C^{\infty}_{W,V}(\mathbb{T})\subset L^{2}_{V}(\mathbb{T})$ such that $u=v$ (V-a.e.).
\end{corollary}

At first sight, it might seem that Corollary \ref{regeigenvect} is stronger than the result provided in \cite[Theorem 6]{keale}. However, by taking an additional step in the argument of \cite[Theorem 6]{keale}, we can arrive at the same conclusion. Indeed, what was proved in \cite[Theorem 6]{keale} is that for every $n\in \mathbb{N}$, there exists $v_n\in C^{n}_{W,V}(\mathbb{T})$ such that $u=v_n$ (V-a.e.). However, since $v_1 = v_n$ $V$-a.e. for every $n\in \mathbb{N}$, this means that they are equal on a dense set of the torus, and since they are right-continuous, they are equal at every point of the torus. Therefore, $v_1\in C_{W,V}^{\infty}(\mathbb{T})$ and $u=v_1$ (V-a.e.). Here, we obtained this result in a direct and more natural manner than in \cite[Theorem 6]{keale}.

\section{Traceability of $\Delta_{W,V}$ under Dirichlet Boundary Conditions}\label{sec:trace}
To simplify notation, throughout this section we assume that $W(0) = 0$. We consider the space $H_{W,V}(\mathbb{T})$ equipped with a Dirichlet-type condition. Fix a point in $\mathbb{T}$ and identify it with zero, using the standard correspondence between the torus and the interval $[0,1)$. Define
\begin{equation*}
H_{W,V,\mathcal{D}}(\mathbb{T}) = \{ f \in H_{W,V}(\mathbb{T}) : f(0) = 0 \}
\end{equation*}
and endow this space with the inner product
\begin{equation*}
\langle f, g \rangle_{W,V,\mathcal{D}} := \int_{\mathbb{T}} D_W^- f \, D_W^- g \, dW.
\end{equation*}
It is straightforward to verify that $\left(H_{W,V,\mathcal{D}}(\mathbb{T}), \langle \cdot, \cdot \rangle_{W,V,\mathcal{D}}\right)$ is a Hilbert space.

For any $f \in H_{W,V}(\mathbb{T})$, it is immediate that
\begin{equation*}
\|f\|_{W,V,\mathcal{D}} \leq \|f\|_{W,V}.
\end{equation*}
On the other hand, by applying Theorem~\ref{sobchar2}, the elementary inequality $(a+b)^2 \leq 2(a^2 + b^2)$, and Jensen's inequality, we deduce the existence of a constant $C > 0$ such that
\begin{equation*}
\|f\|^2_{L^2_V(\mathbb{T})} \leq C \left( (f(0))^2 + \|D_W^- f\|^2_{L^2_W(\mathbb{T})} \right).
\end{equation*}
This demonstrates that the norms $\|\cdot\|_{W,V,\mathcal{D}}$ and $\|\cdot\|_{W,V}$ are equivalent on $H_{W,V,\mathcal{D}}(\mathbb{T})$.

\begin{remark}
As discussed in \cite[Section 7]{keale}, the space $H_{W,V,\mathcal{D}}(\mathbb{T})$ does not depend on $V$ and can be embedded in any $L^{2}_{V}(\mathbb{T})$.
\end{remark}

Let us now recall the definition of the $W$-Brownian bridge introduced in \cite{keale}.

\begin{definition}[$W$-Brownian bridge]
Fix a point in $\mathbb{T}$ and label it as zero. The $W$-Brownian bridge on $[0,1) \cong \mathbb{T}$, denoted by $B_{W,0}(\cdot)$, is the centered Gaussian process with covariance function
\begin{equation*}
\varrho_{W,0}(t,s) = W(t \wedge s) - \frac{W(t)W(s)}{W(1)}, \quad t,s \in [0,1).
\end{equation*}
\end{definition}

The first result we need is the joint measurability of the $W$-Brownian bridge.

\begin{proposition}\label{prop:jointmeas}
The $W$-Brownian bridge $B_{W,0}(\cdot)$ admits a version that is a jointly measurable process.
\end{proposition}

\begin{proof}
It is well known that a continuous version of standard Brownian motion on $[0, W^{-1}(1)]$ exists and is jointly measurable. Moreover, observe that $B_W(t):=B(W(t))$ is a $W$-Brownian motion and, since $W$ is right-continuous, $B_W$ is jointly measurable. By \cite[Proposition 18]{keale}, the $W$-Brownian bridge admits the representation
\begin{equation*}
B_{W,0}(t) = B_W(t) - \frac{W(t)}{W(1)} B_W(1).
\end{equation*}
Since $B_W$ is jointly measurable and the operations involved preserve measurability, this also establishes the joint measurability of $B_{W,0}$.
\end{proof}

By the joint measurability of $B_{W,0}(\cdot)$, we can apply Fubini's theorem to obtain
\begin{equation*}
\mathbb{E}[ \| B_{W,0}\|_{L^2_V(\mathbb{T})}^2] = \int_{\mathbb{T}} \mathbb{E}[B_{W,0}(t)^2]\, dV(t) \leq V(1) W(1) < \infty.
\end{equation*}

In particular, this shows that $B_{W,0}(\cdot) \in L^2(\Omega, L^2_V(\mathbb{T}))$. Now, let us consider the covariance operator of $B_{W,0}(\cdot)$ on the space $L^2_V(\mathbb{T})$. More precisely, let $K:L^2_V(\mathbb{T}) \to L^2_V(\mathbb{T})$ be the operator defined by
\begin{equation*}
Kf(t) = \int_{\mathbb{T}} \varrho_{W,0}(t,s) f(s) dV(s), \quad t \in \mathbb{T}.
\end{equation*}
We now have the following consequence of the above discussion.
\begin{proposition}
The operator $K$ is trace-class.
\end{proposition}
\begin{proof}
By \cite[Theorem 2]{rajput}, since $B_{W,0}(\cdot)$ is a jointly measurable Gaussian process with paths in $L^2_V(\mathbb{T})$, the distribution of $B_{W,0}(\cdot)$ is a Gaussian measure on $L^2_V(\mathbb{T})$ with covariance operator $K$. In particular, by the definition of Gaussian measures, $K$ is trace-class. Furthermore,
\begin{equation*}
\operatorname{tr}(K) = \mathbb{E}[\|B_{W,0}\|_{L^2_V(\mathbb{T})}^2].
\end{equation*}
\end{proof}

Now, let $\Delta_{W,V,\mathcal{D}}$ be the Dirichlet-type Laplacian operator defined by $\Delta_{W,V,\mathcal{D}}: D(\Delta_{W,V}) \cap H_{W,V,\mathcal{D}}(\mathbb{T}) \subset L^2_V(\mathbb{T}) \to L^2_V(\mathbb{T})$ and given by $\Delta_{W,V,\mathcal{D}}f =\Delta_{W,V}f$. Thus, $\Delta_{W,V,\mathcal{D}}$ is a self-adjoint operator.

We will now show that $K = (-\Delta_{W,V,\mathcal{D}})^{-1}$. To this end, observe that an elementary computation shows that for any $f\in L_V^2(\mathbb{T})$, 
\begin{equation}\label{eq:identity_K_operator}
    \begin{aligned}
    \int_{\mathbb{T}} W(t\land s) f(s) dV(s) &= W(t) \int_{[0,1)} f(s) dV(s)\\
    &- \int_{(0,t]} \int_{[0,u)} f(s) dV(s) dW(u)
    \end{aligned}
\end{equation}
and
\begin{equation}\label{eq:identity_K_operator_2}
    \begin{aligned}
    \int_{\mathbb{T}} \frac{W(t)W(s)}{W(1)} f(s) dV(s) &= \frac{W(t)}{W(1)} \int_{[0,1)} W(s) f(s) dV(s).
    \end{aligned}
\end{equation}
Therefore, by \eqref{eq:identity_K_operator} and \eqref{eq:identity_K_operator_2}, we have that for any $f\in L_V^2(\mathbb{T})$,
\begin{align*}
    D_V^+ D_W^- K f(t) &= D_V^+ D_W^- \Bigg[ W(t) \int_{[0,1)} f(s) dV(s) \\
    &- \int_{(0,t]} \int_{[0,u)} f(s) dV(s) dW(u) - \frac{W(t)}{W(1)} \int_{[0,1)} W(s) f(s) dV(s) \Bigg] \\
    &= - f(t),
\end{align*}
which shows that for every $f \in L_V^2(\mathbb{T})$, we have 
\begin{equation*}
\Delta_{W,V} K f = f.
\end{equation*}
Conversely, take any $f \in D(\Delta_{W,V,\mathcal{D}})$. Observe that $f(0) = f(1) = 0$ as well as $\varrho_W(t,0) = 0$ for all $t \in [0,1)$ and $\varrho_W(t,1) = 0$ for all $t \in [0,1)$. Using the integration by parts formula (Theorem \ref{thm:intbyparts}), we obtain
\begin{align*}
K D_V^+ D_W^- f(t) &= \int_{[0,1)} \varrho_{W,0}(t,s) D_V^+ D_W^- f(s) dV(s) \\
&= -\int_{[0,1)} D_W^- \varrho_{W,0}(t,s) D_W^-f(s) dW(s)\\
&= -\int_{[0,1)} D_W^- \left(W(t\land s) - \frac{W(t)W(s)}{W(1)}\right) D_W^-f(s) dW(s) \\
&= - \int_{[0,1)}  1_{[0,t)} (s) D_W^- f(s) dW(s) + \frac{W(t)}{W(1)} \int_{[0,1)} D_W^- f(s) dW(s) \\
&= - f(t) + f(0) + \frac{W(t)}{W(1)} (f(1) - f(0))\\
&= -f(t).
\end{align*}
This shows that for every $f \in D(\Delta_{W,V,\mathcal{D}})$, we have $K D_V^+ D_W^- f = f$. Therefore, $K = (-\Delta_{W,V,\mathcal{D}})^{-1}$.

In particular, this shows that the inverse of the Dirichlet-type Laplacian operator $-\Delta_{W,V,\mathcal{D}}$ is a compact, self-adjoint, trace-class operator. Let 
\begin{equation*}
0 < \lambda_{1,\mathcal{D}} \leq \lambda_{2,\mathcal{D}} \leq \cdots \leq \lambda_{n,\mathcal{D}} \leq \cdots
\end{equation*}
be the eigenvalues of $-\Delta_{W,V,\mathcal{D}}$. We have
\begin{equation}\label{eq:traceability_condition}
\sum_{n=1}^{\infty} \frac{1}{\lambda_{n,\mathcal{D}}} = \operatorname{tr}(K) < \infty.
\end{equation}

Now, since both $-\Delta_{W,V,\mathcal{D}}$ and $-\Delta_{W,V}$ are self-adjoint operators with compact resolvents, they have empty essential spectra. Furthermore, $D(\Delta_{W,V,\mathcal{D}}) \subset D(\Delta_{W,V})$.
%Further, let $B_{W,V}:H_{W,V}(\mathbb{T}) \times H_{W,V}(\mathbb{T}) \to \mathbb{R}$ be the bilinear form associated to $-\Delta_{W,V}$. 
Therefore, by the min-max principle \cite{eschwe2004}, since $H_{W,V,\mathcal{D}}(\mathbb{T}) \subset H_{W,V}(\mathbb{T})$, we have
\begin{align*}
\lambda_{k} &= \min_{\substack{L\subset D(-\Delta_{W,V})\\ \dim(L) = k}} \max_{\substack{f\in L\\ f\neq 0}} \frac{\<f, -\Delta_{W,V} f\>_V}{\<f,f\>_V}\\
&\leq  \min_{\substack{L\subset D(-\Delta_{W,V,\mathcal{D}})\\ \dim(L) = k}} \max_{\substack{f\in L\\ f\neq 0}} \frac{\<f, -\Delta_{W,V,\mathcal{D}} f\>_V}{\<f,f\>_V} = \lambda_{k,\mathcal{D}}.
\end{align*}
In particular, for $k>1$, since $\lambda_1 = 0$ and has multiplicity one, we have
\begin{equation}\label{eq:traceability_condition_2}
\frac{1}{\lambda_{k,\mathcal{D}}} \leq \frac{1}{\lambda_{k}}.
\end{equation}

This shows that the traceability condition obtained in \eqref{eq:traceability_condition} cannot be directly transferred to the Laplacian operator $-\Delta_{W,V}$. In the next section, we will prove in a completely different way that the traceability condition is satisfied for certain fractional powers of the Laplacian operator $-\Delta_{W,V}$. By the relation in \eqref{eq:traceability_condition_2}, we will also obtain, as a special case, a sharper bound for the eigenvalues of $-\Delta_{W,V,\mathcal{D}}$.

\section{Asymptotic behavior of the eigenvalues of $\Delta_{W,V}$}\label{sec:asymptotic_eigenvalues}
	
In this section, we bring together and generalize several results from the previous section to obtain an asymptotic result for the eigenvalues of $\Delta_{W,V}$. First, we highlight the possibility of generalizing the well-known identity $\sin^2(x) + \cos^2(x) = 1$ to the generalized trigonometric functions introduced in Section \ref{sect42}. Next, we consider a generalization of Lemma \ref{taylorbound}, which is important for providing a sharp upper bound (in certain contexts) for the asymptotics of the eigenvalues. This approach leads to a simple abstract result that, to our knowledge, has not yet appeared in the literature.

The following result, inspired by \cite[Theorem 5.3]{trig}, plays an important role in this work.

\begin{proposition}\label{prop11}
	For any $x\in\mathbb{T}$ and any $\alpha\neq 0$, the following fundamental relation holds:
	\begin{equation}\label{relfun}
		C_{W,V}(\alpha,x)C_{V,W}(\alpha,x)+S_{W,V}(\alpha,x)S_{V,W}(\alpha,x)=1
	\end{equation}
\end{proposition}
\begin{proof}
	Let $(\alpha_n)_{n\in\mathbb{N}}$ be given by 
	$$\alpha_{2k}(x)=(G_{2k}(x,x),F_{2k}(x,x))$$
	and 
	$$\alpha_{2k+1}(x)=(F_{2k+1}(x,x),G_{2k+1}(x,x))$$ 
	for $k\geq 0$, where $F_n$ is given by \eqref{term:1} and $G_n$ by \eqref{Gnformula}. Let us denote $\alpha_{n}(x)=(p_{n}(x),q_{n}(x))$. Using the integration by parts formula (see Proposition \ref{intbypartsprop2}), we obtain the following relation:
	\begin{equation}\label{relarela}
		\sum_{j=0}^{2k}(-1)^{j}q_{j}(x)p_{2n-j}(x)=0.
	\end{equation}
	Finally, applying (\ref{relarela}) to 
	$$C_{W,V}(\alpha,x)C_{V,W}(\alpha,x)+S_{W,V}(\alpha,x)S_{V,W}(\alpha,x)=1+\sum_{n=1}^{\infty}(-1)^{n}\alpha^{2n}\sum_{k=0}^{2n}(-1)^{k}q_{k}(x)p_{2n-k}(x),$$
	the relation (\ref{relfun}) follows.
\end{proof}

We now obtain a log-asymptotically sharp estimate related to $F_{2n}$ and $G_{2n}$, which will be fundamental in our study of the asymptotic behavior of the eigenvalues of $\Delta_{W,V}$.

\begin{proposition}\label{prop22}
	There exists a constant $C>0$ such that 
	\begin{equation}\label{kjrsas}
		|F_{2n}(1,1)+G_{2n}(1,1)|\le  \dfrac{C^{n}}{(n!)^{2}}.
	\end{equation}
\end{proposition}
\begin{proof}
	We begin by observing that the functions $\frac{W^{n+1}}{n+1}$ and $\frac{V^{n+1}}{n+1}$ are, respectively, $W$-absolutely continuous and $V$-absolutely continuous in the sense of \cite{pouso}. Using the derivatives as defined in \cite{pouso}, it follows that
	$$\left(\frac{W^{n+1}}{n+1}\right)'_{W}(x)=\frac{1}{n+1}\sum_{j=0}^{n}W(x-)^{j}W(x)^{n-j}$$
	and 
	$$\left(\frac{V^{n+1}}{n+1}\right)'_{V}(x)=\frac{1}{n+1}\sum_{j=0}^{n}V(x+)^{j}V(x)^{n-j}.$$ 
	Using $W(x)\geq W(x-)$ and $V(x+)\geq V(x)$ together with \cite[Theorem 5.4]{pouso}, we obtain the following inequalities:
	\begin{equation}\label{estimaW}
		\frac{[W(x)]^{n+1}}{n+1}\geq\int_{(0,x]}W^{n}(\xi-)dW(\xi)
	\end{equation}
	and 
	\begin{equation}\label{estimaV}
		\frac{[V(x)]^{n+1}}{n+1}\geq\int_{[0,x)}V^{n}(\xi)dV(\xi).
	\end{equation}
	Moreover, notice that
	\begin{equation*}
		\begin{split}
			F_{2}(x,x)=\int_{(0,x]}\int_{[0,s)} dV(\xi)dW(s)=\int_{(0,x]}V(s)dW(s)\le V(x)W(x).
		\end{split}
	\end{equation*}
	So, 
	\begin{equation*}
		\begin{split}
			F_{4}(x,x) &=\int_{(0,x]}\int_{[0,s)} F_{2}(\xi,\xi)dV(\xi)dW(s) \le \int_{(0,x]}\int_{[0,s)} V(\xi)W(\xi) dV(\xi)dW(s)\\
			&\le \int_{(0,x]}W(s-) \left(\int_{[0,s)} V(\xi)dV(\xi)\right)dW(s)\le \int_{(0,x]}W(s-)\frac{V(s)^{2}}{2}dW(s)\\
			&\le \frac{V(x)^{2}}{2}\int_{(0,x]}W(s-)dW(s) \le \frac{V(x)^{2}}{2}\frac{W(x)^{2}}{2}
		\end{split}
	\end{equation*}
	Therefore, after applying the same idea $n$ times using the inequalities (\ref{estimaW}) and (\ref{estimaV}) to $F_{2n}(x,x)$, it follows that there exists a constant $A>0$ such that $|F_{2n}(1,1)|\le \frac{A^{n}}{(n!)^2}$. Proceeding similarly for $G_{2n}$ (using the corresponding expression for $G_{k}$ as in Corollary \ref{rmk111}) and successively applying (\ref{estimaW}) and (\ref{estimaV}) to $G_{2n}(x,x)$, we obtain the existence of $B>0$ such that $|G_{2n}(1,1)|\le \frac{B^{n}}{(n!)^2}$. Therefore, inequality (\ref{kjrsas}) follows if we take $C=2(A+B)$.
\end{proof}

We are now in a position to state the main result of this section, which provides a lower bound on the growth of the eigenvalues of $\Delta_{W,V}$.

\begin{theorem}\label{prop6}
	Let $W,V:\mathbb{R}\to\mathbb{R}$ be any strictly increasing functions satisfying the periodic condition \eqref{eqsmedida}, and let $\Delta_{W,V}:\mathcal{D}_{W,V}(\mathbb{T})\subset L^2_V(\mathbb{T})\to L^2_V(\mathbb{T})$ be their induced $W$-$V$-Laplacian. Let $\{\lambda_{i}\}_{i\geq1}$ be the sequence of non-negative eigenvalues (counted according to their multiplicity) of $\Delta_{W,V}$. Then, there exists $\rho\in (0,\frac{1}{2}]$ such that 
	\begin{equation}\label{eigenineqeq}
		Cn^{1/\rho}\le \lambda_{n}
	\end{equation}
	for some constant $C>0$. Moreover, 
	\begin{equation}\label{convsumeigen}
		\sum_{i=1}^{\infty}\dfrac{1}{\lambda_{i}^s}<\infty
	\end{equation}
	for $s>\rho$. Furthermore, the bound in equation \eqref{eigenineqeq} is sharp in the sense that it is attained for some $W$ and $V$.
\end{theorem}
\begin{proof}
	From Proposition \ref{prop11}, the system (\ref{bcondi}) has a non-trivial solution $(a_{i},b_{i})$ if $\lambda_{i}>0$ satisfies
	$$2=C_{W,V}(\sqrt{\lambda_{i}},1)+C_{V,W}(\sqrt{\lambda_{i}},1).$$
	That is, $\lambda_{i}$
	are positive roots of the entire function
	$$f(z)=-2+\sum_{n\geq 0}(-1)^{n}z^{n}(F_{2n}(1,1)+G_{2n}(1,1))=\sum_{n\geq 1}(-1)^{n}z^{n}(F_{2n}(1,1)+G_{2n}(1,1)).$$
	Therefore, the eigenvalues of $\Delta_{W,V}$ are zeros of $f(z)$.
	The inequality $$|F_{2n}(1,1)+G_{2n}(1,1)|\le\dfrac{C^{n}}{(n!)^2}$$
	implies that the order of growth $\rho$ of the series $f(z)$ satisfies $0\le\rho\le \dfrac{1}{2}$. Indeed, using the Stirling's approximation, we have 
	$$\rho=\limsup_{n\to\infty}\dfrac{n\ln n}{-\ln \left|(-1)^{n}[F_{2n}(1,1)+G_{2n}(1,1)]\right|}\le \limsup_{n\to\infty}\dfrac{n\ln n}{-\ln\left[\dfrac{C^{n}}{(n!)^2}\right]}=\dfrac{1}{2}.$$ 
	Now, let $\mathcal{Z}=\{z_{i}\}_{i\in\mathbb{N}}$ be the zeros of $f$, indexed according to their multiplicity and ordered by their moduli:
	$$0<|z_{1}|\le |z_{2}| \le \ldots .$$
	Let $n(r)$ be the number of zeros of $f$ whose moduli are less than or equal to $r$, that is,
	$$n(r)=\#\{i\in\mathbb{N}; |z_{i}|<r\}.$$
	Then, by \cite[Chapter 5, Theorem 2.1]{stein}, for any $\beta\ge\rho$ there exists a constant $C>0$ such that for sufficiently large $r>0$,  
	$$n(r)\le Cr^{\beta}.$$
	Now, observe that $\{\lambda_{i}\}_{i\in\mathbb{N}}\subset \mathcal{Z}$, and by \eqref{asymplambda}, we have $\lambda_{i}\to \infty$ as $i\to\infty$. Therefore, $\lim_{i\to\infty}|z_{i}|=+\infty$, which implies that $\rho\neq 0$. The above results also show, in particular, that $\rho\not\in\mathbb{Z}$.
	By \cite[Chapter 5, Theorem 1]{levin}, the order $\rho$ is equal to the convergence exponent of $\mathcal{Z}$. In particular, there exists $C>0$ such that
	\begin{equation}\label{estim11}
		\limsup_{r\to+\infty}\dfrac{n(r)}{r^{\rho}}=\limsup_{n\to+\infty}\dfrac{n}{|z_{n}|^{\rho}}\le C.
	\end{equation}
	From (\ref{estim11}), there exists a subsequence $n_{k}$ such that
	\begin{equation}\label{lowerboundeq}
		\dfrac{n_{k}}{C}\le \lambda_{k}^{\rho}.
	\end{equation}
	By the definition of subsequence, we have $n_{k}\geq k$. Combining \eqref{lowerboundeq} and $n_k\geq k$, we obtain the inequality (\ref{eigenineqeq}). Finally, since $\rho$ is equal to the convergence exponent of $\mathcal{Z}$, we directly obtain that if $s>\rho$, then
	\begin{equation}\label{convseries}
		\sum_{i\geq 1}\dfrac{1}{\lambda_{i}^{s}}<\infty.
	\end{equation}
	The final claim follows from Weyl's asymptotics (see, for instance, \cite[Theorem 6.3.1]{davies}). Indeed, Weyl's asymptotics yields that the exact order for the Laplacian, which corresponds to the case $W(x) = V(x) = x$, is $\lambda_n \propto n^{2}$.
\end{proof}

In Section \ref{sec:applications}, we will see that Theorem \ref{prop6} can be used to establish the traceability of negative fractional powers of the generalized Laplacian.
    
\begin{remark}\label{remark_afterthem}
	In a similar spirit, one may consider a generalized formulation of the problem over compact intervals of $\mathbb{R}$ for measures that do not necessarily satisfy the periodicity condition \eqref{eqsmedida}. This is precisely the setting studied in \cite{trig}. Within the same framework, \cite{Freiberg_2005} observed that there are specific situations in which the parameter $\rho$ appearing in Theorem \ref{prop6} is strictly less than $1/2$. We point out that, unlike the case studied in \cite{Freiberg_2005}, the result of Theorem \ref{prop6} does not specify the exact value of $\rho$.
\end{remark}

\begin{remark}
	We also note that the result stated in Theorem \ref{prop6} extends to the case where the supports of the measures $dW$ and $dV$ are not the whole torus (or the entire interval $[0,1]$), as in the problems studied in \cite{Freiberg_2005} or \cite{trig}. The key idea is that in these cases, the eigenvalues are also characterized as roots of entire functions. 
\end{remark}

An immediate consequence of Theorem \ref{prop6} is its counterpart for the eigenvalues of the Dirichlet Laplacian $\Delta_{W,V,\mathcal{D}}$ introduced in Section \ref{sec:trace}:
\begin{corollary}
	Let $\{\lambda_{i,\mathcal{D}}\}_{i\geq1}$ be the sequence of non-negative eigenvalues (counted according to their multiplicity) of $\Delta_{W,V,\mathcal{D}}$. Then, there exists $\rho\in (0,\frac{1}{2}]$ such that
	\begin{equation*}
		Cn^{1/\rho}\le \lambda_{n, \mathcal{D}}
	\end{equation*}
	for some constant $C>0$. Moreover,
	\begin{equation*}
		\sum_{i=1}^{\infty}\dfrac{1}{\lambda_{i,\mathcal{D}}^s}<\infty
	\end{equation*}
	for $s>\rho$. 
\end{corollary}

To summarize, the strategy developed in this section can be used to study the abstract asymptotics of eigenvalues of compact operators.

\begin{theorem}\label{theorem:7colofmain}
	Let $A:H\to H$ be an injective compact operator defined on a Banach space $H$ whose set of eigenvalues $\mathcal{Z}:=\{\lambda_{i}\}_{i\in\mathbb{N}}$ converges to $0$. If $\mathcal{Z}$ is a subset of the zeros of an entire function $f(z)=\sum_{n=0}^{\infty} a_{n}z^{n}$ such that $$\rho:=-\operatorname{limsup}_{n \rightarrow+\infty} \frac{n \ln n}{\ln \left|a_n\right|}<\infty,$$
	then, if $n(r):=\#\{i\in\mathbb{N};|\lambda_{i}|<r\}$, the following assertions hold: 
	\begin{enumerate}
		\item There exists a constant $C>0$ such that $n(r)\le Cr^{\rho}$;
		\item For $s>\rho$, we have $\sum_{i}\frac{1}{|\lambda_{i}|^{s}}<\infty$.
	\end{enumerate} 
\end{theorem}
\begin{proof}
	The proof of this result is contained in the proof of Theorem \ref{prop6}.
\end{proof}

As previously mentioned, this result allows us to improve the convergence of the series $\sum_{i} \frac{1}{\lambda_{i,\mathcal{D}}^{s}}$ for all $s > \frac{1}{2}$, thereby establishing the traceability of the negative $s$-fractional powers of $\Delta_{W,V,\mathcal{D}}$. In particular, for the case $s = 1$, this recovers the traceability of the operator $(-\Delta_{W,V,\mathcal{D}})^{-1}$, which was already proved in Section~\ref{sec:trace}. In the same spirit, we have the following result:
\begin{theorem}\label{summarizationthem}
	Let $A:H\to H$ be a trace-class operator defined on a separable Hilbert space. Furthermore, let $\Lambda^{k}(A)$ denote the $k$-th exterior power of the operator $A$. Then, for every 
	$$s>-\limsup _{n \rightarrow+\infty} \frac{n \ln n}{\ln \left|\text{tr}\left[\Lambda^{k}(A)\right]\right|}$$ 
	we have $\sum_{i\in\mathbb{N}} \frac{1}{|\lambda_{i}|^{s}} <\infty$.
\end{theorem}
\begin{proof}
	Since $A$ is a trace-class operator, it follows from \cite[Theorem 5.2]{Gohberg_2000} that $z\mapsto \det(I-zA)=\sum_{k=0}^{\infty}z^{k}\mbox{tr}\left[\Lambda^{k}(A)\right]$ is an entire function whose zeros are exactly the eigenvalues of $A$. Consequently, using Theorem \ref{theorem:7colofmain}, the result follows.
\end{proof}

\section{Applications}\label{sec:applications}

In this section, we present several applications of the results derived in the previous sections. These applications are organized into three distinct subsections. The first subsection is devoted to the study of higher-order Sobolev spaces and the nuclearity of the space $C^{\infty}_{W,V}(\mathbb{T})$. We introduce the fractional-order Sobolev spaces, present some compatibility results, and discuss applications to the traceability of $(I+\Delta_{W,V})^{-s}$. We then show how this framework leads to a regularity result for fractional elliptic problems. The second subsection is dedicated to the existence of solutions for fractional stochastic differential equations driven by $V$-Gaussian white noise. In the third and final subsection, we take advantage of our characterization of $C^{\infty}_{W,V}(\mathbb{T})$ as a nuclear space to establish existence and uniqueness of solutions for parabolic stochastic partial differential equations.

\subsection{Nuclearity of $C^{\infty}_{W,V}(\mathbb{T})$}
Recall that the usual higher-order Sobolev spaces play a crucial role in the study of fractional Sobolev spaces. The situation in this context is no different. Based on the characterizations introduced in Section \ref{sec:review}, we begin by introducing the higher-order $W$-$V$-Sobolev spaces.
\begin{definition}\label{highersob}
	Given $k\in\mathbb{N}$, we define the $W$-$V$-Sobolev space of order $k$ as
	$$H^{k}_{W,V}(\mathbb{T})=\left\{f\in L^{2}_{V}(\mathbb{T}) : \exists D_{W,V}^{(n)}f\in L^{2}_{\kappa(n)}(\mathbb{T})\,\forall n=1,\ldots,k\right\},$$
	where $D_{W,V}^{(n)}$ is considered in the weak sense (see Definition \ref{defWleft}) and
	\begin{equation}\label{kappa_order}
		\kappa(n):=\left\{ \begin{array}{ll}
			W,\;\mbox{if}\; n\; \mbox{is odd}; \\
			V,\;\mbox{if}\; n\; \mbox{is even}.\end{array} \right. 
	\end{equation}
	Furthermore, the space $H^{k}_{W,V}(\mathbb{T})$, endowed with the norm
	$$\|f\|_{k,W,V}(\mathbb{T}):=\left(\|f\|_{V}^{2}+\sum_{i=1}^{k}\|\partial^{(i)}_{W,V}f\|_{\kappa(i)}^{2}\right)^{1/2},$$
	is a Hilbert space.
\end{definition}

\begin{remark}
	We note that, in order to define the weak $V$-derivatives, we must use the space $C^{\infty}_{W,V}(\mathbb{T})$—the space of regular functions associated with the adjoint problem—instead of $C^{\infty}_{V,W}(\mathbb{T})$.
\end{remark}

In our next result, we generalize the relation \eqref{fourierHWV} by characterizing $H^{k}_{W,V}(\mathbb{T})$ in terms of its Fourier coefficients.
\begin{theorem}\label{fouriersob}
	We have the following characterization:
	$$H^{k}_{W,V}(\mathbb{T})=\left\{f\in L^{2}_{V}(\mathbb{T}) : f=\alpha_{0}+\sum_{i=1}^{\infty}\alpha_{i}\nu_{i}\ \text{and}\ \sum_{i=1}^{\infty}\lambda_{i}^{k}\alpha_{i}^2<\infty\right\}.$$
\end{theorem}
\begin{proof}
	Let $k$ be odd and $f=\alpha_{0}+\sum_{i=1}^{\infty}\alpha_{i}\nu_{i}$. By Lemma \ref{Wderivclosed}, we have
	$$\sum_{i=1}^{\infty}\alpha_{i}\lambda_{i}^{\frac{k-1}{2}}\sqrt{\lambda_{i}}\frac{D^{-}_{W}\nu_{i}}{\sqrt{\lambda_i}}=D^{(k)}_{W,V}f\in L^{2}_{W}(\mathbb{T}).$$
	Indeed, by Parseval's identity,
	$$ \sum_{i=1}^{\infty}\alpha_{i}^2\lambda_{i}^{k} = \|D^{(k)}_{W,V}f\|^{2}_{W} < \infty.$$
	Conversely, let $f = \alpha_{0} + \sum_{i=1}^{\infty} \alpha_{i} \nu_{i}$ and suppose $\sum_{i=1}^{\infty} \lambda_{i}^{k} \alpha_{i}^2 < \infty$. Recall from \eqref{asymplambda} that $\lambda_{i} \to \infty$ as $i \to \infty$. This implies that for large $i \in \mathbb{N}$, $\lambda_{i} > 1$. Therefore, if $\sum_{i=1}^{\infty} \alpha_{i}^2 \lambda_{i}^{k} < \infty$, then $\sum_{i=1}^{\infty} \alpha_{i}^2 \lambda_{i}^{j} < \infty$ for all $j \le k$. Now, let $(f_k)_{k \in \mathbb{N}}$ be the sequence given by
	$$f_{k} = \alpha_{0} + \sum_{i=1}^{k} \alpha_{i} \nu_{i}.$$
	We have $f_{k} \in C^{\infty}_{W,V}(\mathbb{T})\subset H^{k}_{W,V}(\mathbb{T})$ and $D^{(n)}_{W,V} f_{k}$ is Cauchy in $L^{2}_{\kappa(n)}(\mathbb{T})$ for $n = 1, \ldots, k$. By the completeness of $H_{W,V}^k(\mathbb{T})$, there exists $\widetilde{f}\in H^{k}_{W,V}(\mathbb{T})$ such that $f_{k} \to \widetilde{f}$ with respect to the norm $\|\,\cdot\,\|_{k,W,V}$. By uniqueness of the $L^2_V(\mathbb{T})$-limit, we have $\widetilde{f} = f$. Thus, $f\in H^{k}_{W,V}(\mathbb{T})$.
\end{proof}

\begin{corollary}\label{approxsmoothhigh}
	The following characterization of the space $H^{k}_{W,V}(\mathbb{T})$ holds: 
	\begin{equation*}\label{characHk}
		H^{k}_{W,V}(\mathbb{T})=\overline{C^{\infty}_{W,V}(\mathbb{T})}^{\|\,\cdot\,\|_{k,W,V}}.
	\end{equation*}
\end{corollary}

We have the following strong regularity for these higher-order $W$-$V$-Sobolev spaces:

\begin{theorem}\label{regularity}
	For every $k \geq 1$,
	$$H^k_{W,V}(\mathbb{T}) \subset C_{W,V}^{k-1}(\mathbb{T}),$$
	where $C^0_{W,V}(\mathbb{T}) = \{f:\mathbb{T}\to\mathbb{R} : \text{$f$ is càdlàg and $D(f)\subset D(V)$}\}$, and $D(g)$ denotes the set of discontinuity points of $g$.
\end{theorem}
\begin{proof}
	The case $k=1$ follows directly from \cite[Theorem 2]{keale}. By induction on $k$, and using \cite[Theorem 2]{keale} again, we obtain the result for any $k \geq 1$.
\end{proof}

As an immediate consequence of Theorem \ref{regularity}, we have the following corollary:
\begin{corollary}\label{hinfinityreg}
	The following equality of sets holds:
	$$\bigcap_{k=1}^{\infty}H^{k}_{W,V}(\mathbb{T}) =: H^{\infty}_{W,V}(\mathbb{T}) = C^\infty_{W,V}(\mathbb{T}).$$
\end{corollary}

In the following result, we use Theorem \ref{prop6} and the results established in this section to conclude that $C^{\infty}_{W,V}(\mathbb{T})$ is a nuclear space (see Appendix \ref{app:spde_kallianpur} for the definition of nuclear spaces).

\begin{theorem}
	The space $C^{\infty}_{W,V}(\mathbb{T})$ is a nuclear space. 
\end{theorem}
\begin{proof}
	From Corollary \ref{hinfinityreg}, we have $C^{\infty}_{W,V}(\mathbb{T}) \subset H^{\infty}_{W,V}(\mathbb{T})$. To show that $C^{\infty}_{W,V}(\mathbb{T})$ is a nuclear space, we need to show that for each $n \geq 0$, there exists $m > n$ such that the natural inclusion $i_{m,n}: H^{m}_{W,V}(\mathbb{T}) \to H^{n}_{W,V}(\mathbb{T})$ is a Hilbert-Schmidt operator. Indeed, take any $m>n$. Denoting by $\nu_{i}$ the eigenfunctions of $\Delta_{W,V}$ associated with the eigenvalues $\lambda_{i}$, the family $\left\{\frac{\nu_{i}}{\lambda^{m}_{i}}\right\}_{i\in\mathbb{N}}$ is an orthonormal basis of $H^{m}_{W,V}(\mathbb{T})$. Consequently, taking $\rho$ as in Theorem \ref{prop6}, we have
	\begin{equation*}
		\sum_{i=1}^\infty \left\|\frac{\nu_{i}}{\lambda^{m}_{i}}\right\|_{n,W,V}^2 = \sum_{i=1}^\infty \lambda_{i}^{2(n-m)} \le  \sum_{i=1}^\infty i^{\frac{2}{\rho}(n-m)} < \infty.
	\end{equation*}
	where the last convergence holds because $\frac{2}{\rho}(m-n)\geq 4(m-n)>1$. 
\end{proof}

For $s\geq 0$, let $(I-\Delta_{W,V})^{s}:H^{2s}_{W,V}(\mathbb{T})\to L^{2}_{V}(\mathbb{T})$, where
\begin{equation*}\label{H2s}
	H^{2s}_{W,V}(\mathbb{T}) = \left\{f\in L^{2}_{V}(\mathbb{T}) : \sum_{i=0}^\infty \gamma_{i}^{2s}(f,\nu_{i})^{2}<\infty\right\},
\end{equation*}
$\gamma_{i}=1+\lambda_{i}$, and 
\begin{equation*}
	(I-\Delta_{W,V})^{s}f = \sum_{i=0}^\infty \gamma_{i}^{s}(f,\nu_{i})\nu_{i}, \quad \forall f\in H^{2s}_{W,V}(\mathbb{T}),
\end{equation*}
be the $s$-fractional power of the operator $I-\Delta_{W,V}$. We note that the space $H^{s}_{W,V}(\mathbb{T})$ is essentially a generalization of the higher-order Sobolev spaces as presented in Definition \ref{highersob}. Indeed, when $s=k\in\mathbb{N}$, we have the identification $H^{s}_{W,V}(\mathbb{T})=H^{k}_{W,V}(\mathbb{T})$. A straightforward consequence of this fact is that
\begin{equation}\label{intersection}
	\bigcap_{s>0}H^{s}_{W,V}(\mathbb{T}) = C^{\infty}_{W,V}(\mathbb{T}).
\end{equation}
Moreover, in view of \eqref{intersection} and the identification of the spaces $H^{k}_{W,V}(\mathbb{T})$ as the domains of powers of the operator $I-\Delta_{W,V}$, it follows that \eqref{convseries} is also necessary for the nuclearity of $C^{\infty}_{W,V}(\mathbb{T})$; see, for instance, \cite[Section 3]{kallianpur_1987}. 

For $s \geq 0$, define $H^{-s}_{W,V}(\mathbb{T}) := (H^s_{W,V}(\mathbb{T}))^\ast$, the dual of $H^s_{W,V}(\mathbb{T})$. The following proposition characterizes $H^{-s}_{W,V}(\mathbb{T})$:
\begin{proposition}\label{dualsob}
For $s \geq 0$,
\begin{equation}\label{formuladual}
H^{-s}_{W,V}(\mathbb{T}) \cong \left\{f = \sum_{i=1}^\infty \alpha_i \nu_i : \sum_{i=1}^\infty \gamma_i^{-s} \alpha_i^2 < \infty \right\},
\end{equation}
with norm $\|f\|_{H^{-s}_{W,V}(\mathbb{T})}^2 = \sum_{i=1}^\infty \gamma_i^{-s} \alpha_i^2$ and dual pairing $(f, g) = \sum_{i=1}^\infty \alpha_i \langle \nu_i, g \rangle_V$, where $\alpha_i = (f, \nu_i)$.
\end{proposition}
\begin{proof}
Let $f \in H^{-s}_{W,V}(\mathbb{T})$. By the Riesz representation theorem, there exists $u \in H^s_{W,V}(\mathbb{T})$ such that for all $g \in H^s_{W,V}(\mathbb{T})$,
\begin{equation*}
(f, g) = \langle u, g \rangle_{H^s_{W,V}(\mathbb{T})}.
\end{equation*}
Write $u = \sum_{i=1}^\infty \beta_i \nu_i$ with $\sum_{i=1}^\infty \gamma_i^s \beta_i^2 < \infty$. Then,
$(f, g) = \sum_{i=1}^\infty \gamma_i^s \beta_i \langle \nu_i, g \rangle_V$.
Setting $\alpha_i = \gamma_i^s \beta_i$, we have $(f, \nu_i) = \gamma_i^s \beta_i$ and
$$\sum_{i=1}^\infty \gamma_i^{-s} \alpha_i^2 = \sum_{i=1}^\infty \gamma_i^s \beta_i^2 < \infty.$$
Thus,
\begin{equation*}
\|f\|_{H^{-s}_{W,V}(\mathbb{T})}^2 = \|u\|_{H^s_{W,V}(\mathbb{T})}^2 = \sum_{i=1}^\infty \gamma_i^s \beta_i^2 = \sum_{i=1}^\infty \gamma_i^{-s} \alpha_i^2.
\end{equation*}
The converse follows similarly.
\end{proof}

We also have the following result regarding the trace-class property of the fractional powers of the generalized Laplacian:

\begin{proposition}
    Let $\rho\in (0,1/2]$ be as in Theorem \ref{prop6}. If $s>\rho$, then $\mathcal{A}^{-s}=\left(I-\Delta_{W,V}\right)^{-s}$ is a trace-class operator. Similarly, for the fractional powers of the Dirichlet generalized Laplacian: if $s>\rho$, then $\left(-\Delta_{W,V,\mathcal{D}}\right)^{-s}$ is a trace-class operator.
\end{proposition}

\begin{proof}
    This is an immediate consequence of Theorem \ref{prop6}.
\end{proof}

We note that in the context of fractional Sobolev spaces, as a byproduct of the proof of Theorem \ref{prop6}, \eqref{convseries}, and Proposition \ref{dualsob}, the following proposition easily follows:
\begin{proposition}\label{HSinclusion}
	Let $\rho\in (0,1/2]$ be as in Theorem \ref{prop6}. For $t\in\mathbb{R}$ and $s>\rho$, the inclusion $i:H^{t}_{W,V}(\mathbb{T})\to H^{t-s}_{W,V}(\mathbb{T})$ is trace-class. In particular, for any $s>1/2$, the inclusion $i:H^{t}_{W,V}(\mathbb{T})\to H^{t-s}_{W,V}(\mathbb{T})$ is trace-class.
\end{proposition}

Define the bilinear form $B_{s}(u,v) := \<u,v\>_{ H_{W,V}^s}$ for $u, v \in  H_{W,V}^s(\mathbb{T})$. If $f\in L^{2}_{V}(\mathbb{T})$, we say that $u \in  H_{W,V}^s(\mathbb{T})$ is a weak solution to the equation $(I-\Delta_{W,V})^{s}u = f$ if 
\begin{equation}\label{weaksol}
	B_{s}(u, v) = \<f, v\>_V, \quad \forall v \in  H_{W,V}^s(\mathbb{T}).
\end{equation}
We have the following regularity result for weak solutions of generalized fractional elliptic equations.
\begin{proposition}\label{prop99}
	Fix $s\geq0$ and $m>0$. If $f\in  H^m_{W,V}(\mathbb{T})$, then the function $u\in H_{W,V}^s(\mathbb{T})$ given by the unique weak solution of 
	\begin{equation}\label{weaksol_eq}
		(I-\Delta_{W,V})^{s}u=f
	\end{equation}
	satisfies $u\in  H_{W,V}^{m+2s}(\mathbb{T})$. In particular, for every fixed $s\geq 0$, if $f\in C^{\infty}_{W,V}(\mathbb{T})$, the solution $u$ of \eqref{weaksol_eq} belongs to $C^{\infty}_{W,V}(\mathbb{T})$.
\end{proposition}
\begin{proof}
	Let $f \in  H^m_{W,V}(\mathbb{T})$. Consider the orthonormal basis $\{e_{i,L}\}_{i\in\mathbb{N}}$ of $L^2_V(\mathbb{T})$ consisting of eigenfunctions of $L_{W,V}$ with corresponding eigenvalues $\gamma_{i,L}$. Write
	\begin{equation*}
		f = \sum_{i\in\mathbb{N}} \beta_i e_{i,L}, \qquad u = \sum_{i\in\mathbb{N}} \alpha_i e_{i,L}.
	\end{equation*}
	By the definition of weak solution, we have, in particular, that
	\begin{equation}\label{weakdef}
		B_{s}(u, e_{j,L}) = \<f, e_{j,L}\>_V, \quad \forall j\in\mathbb{N}.
	\end{equation}
	On the other hand, we have $\<f, e_{j,L}\>_V = \beta_j$ and 
	\begin{equation*}	
		B_{s}(u, e_{j,L}) = \<u, e_{j,L}\>_{ H_{W,V}^s} = \gamma_{j,L}^s \<u, e_{j,L}\>_V = \gamma_{j,L}^s \alpha_j.
	\end{equation*}
	Thus, \eqref{weakdef} yields $\gamma_{j,L}^s \alpha_j = \beta_j$, and consequently $\alpha_j = \gamma_{j,L}^{-s} \beta_j$. Since $f \in  H^m_{W,V}(\mathbb{T})$, we have $\sum_{i\in\mathbb{N}} \gamma_{i,L}^m \beta_i^2 < \infty$. Therefore,
	\begin{equation*}
		\sum_{i\in\mathbb{N}} \gamma_{i,L}^{m+2s} \alpha_i^2 = \sum_{i\in\mathbb{N}} \gamma_{i,L}^{m+2s} (\gamma_{i,L}^{-s} \beta_i)^2 = \sum_{i\in\mathbb{N}} \gamma_{i,L}^m \beta_i^2 < \infty,
	\end{equation*}
	which shows that $u \in  H_{W,V}^{m+2s}(\mathbb{T})$.
\end{proof}

\subsection{Fractional Stochastic Differential Equations Driven by $V$-Gaussian White Noise}\label{app:spde_vgaussian}

The main concept required in this section is the notion of $V$-Gaussian white noise $\dot{B}_{V}$ (see Definition \ref{gaussianwhitenoiseV} below). This process is closely related to the so-called $W$-Brownian motion introduced in \cite{keale}. For the applications we have in mind, it is natural to consider $V$ instead of $W$ because the initial space of $I-\Delta_{W,V}$ is $L^{2}_{V}(\mathbb{T})$. Unlike the equations considered in \cite[Section 8]{keale}, we are interested in defining and studying fractional stochastic differential equations of the form
\begin{equation}\label{spde_WV}
    (I-\Delta_{W,V})^{s}u=\dot{B}_{V},\quad \text{on } \mathbb{T}.
\end{equation}
The equation in (\ref{spde_WV}) can be viewed as a generalization of the Whittle-Matérn stochastic differential equation studied in \cite{Bolin_2023, lindgren2011explicit, Bolin_2024, Korte_Stapff_2025}, and has been used extensively in both spatial (see, e.g., \cite{LINDGREN2022100599} for a recent review and comprehensive list of applications) and spatio-temporal modeling of random fields \cite{Lindgren_Bakka_Bolin_Krainski_Rue_2024}. Solutions to \eqref{spde_WV} are known as Whittle–Matérn fields (see, e.g., \cite{Bolin_2023}) or Bessel fields (see, e.g., \cite{Pitt_2003}). The main motivation for considering \eqref{spde_WV} in the sense of $W$-$V$ generalized derivatives is that the sample paths of the solution field $u$ naturally exhibit jumps, which may provide a more effective way to model data with discontinuities.

\begin{definition}\label{gaussianwhitenoiseV}
The $V$-Gaussian white noise is the $L^2_V(\mathbb{T})$-isonormal Gaussian process $\{\dot{B}_V(h) : h \in L^2_V(\mathbb{T})\}$ on a probability space $(\Omega, \mathcal{F}, P)$, such that for all $g, h \in L^2_V(\mathbb{T})$,
\begin{equation*}
E[\dot{B}_V(g)\dot{B}_V(h)] = \int_{\mathbb{T}} gh\, dV.
\end{equation*}
\end{definition}

\begin{remark}\label{whitenoiselinear}
$\dot{B}_V$ is a linear isometry from $L^2_V(\mathbb{T})$ into a closed subspace of $L^2(\Omega, \mathcal{F}, P)$ (see \cite[Chapter 1]{nualart}).
\end{remark}

The existence of $V$-Gaussian white noise follows from Kolmogorov's extension theorem, since the inner product $\<g,h\>_V$ defines a well-defined kernel. The following process is a natural generalization of the $W$-Brownian motion and is closely related to the $V$-Gaussian white noise.

\begin{definition}\label{W-brownian-in-law2}
A process $B_V(t)$ is a $V$-Brownian motion in law if:
\begin{enumerate}
\item $B_V(0) = 0$ a.s.;
\item For $t > s$, $B_V(t) - B_V(s)$ is independent of $\sigma(B_V(u): u \leq s)$;
\item For $t > s$, $B_V(t) - B_V(s) \sim N(0, V(t) - V(s))$.
\end{enumerate}
If $B_V(\cdot)$ has c\`adl\`ag paths, it is a $V$-Brownian motion.
\end{definition}

Notice that the $V$-Brownian motion in law can be recovered from the $V$-Gaussian white noise by
\begin{equation}\label{brownian_from_noise}
B_V(t) = \dot{B}_V(\boldsymbol{1}_{[0,t)}).
\end{equation}
Moreover, from \cite[Proposition 10]{keale}, we have that $B_V(\cdot)$ admits a modification that is a $V$-Brownian motion. On the other hand, for a simple function $f = \sum_{i=1}^n \alpha_i \boldsymbol{1}_{I_i}$, we define the stochastic integral with respect to $B_V(\cdot)$ as
\begin{equation*}
\int_{\mathbb{T}} f(s) dB_V(s) = \dot{B}_V(f).
\end{equation*}
For $h \in L^2_V(\mathbb{T})$, the stochastic integral is defined as the $L^2$-limit of integrals of simple approximations:
\begin{equation}\label{stochintegral}
\int_{\mathbb{T}} h\, dB_V = \lim_{n\to\infty} \dot{B}_V(h_n),
\end{equation}
where $h_n \to h$ in $L^2_V(\mathbb{T})$. Therefore,
\begin{equation}\label{whitenoisestochint}
\dot{B}_V(h) = \int_{\mathbb{T}} h\, dB_V.
\end{equation} 
Since $\dot{B}_V$ is a linear isometry, it follows that
\begin{equation}\label{isometryintegral}
E\left[\left(\int_{\mathbb{T}} h\, dB_V\right)^2\right] = \int_{\mathbb{T}} h^2\, dV
\end{equation}
for all $h \in L^2_V(\mathbb{T})$. 

\begin{remark}\label{whitenoiseL2Sobolev}
By \cite[Proposition 12]{keale} and \eqref{whitenoisestochint}, the $V$-Gaussian white noise on $L^2_V(\mathbb{T})$ coincides with the pathwise $V$-Gaussian white noise on $H_{W,V,\mathcal{D}}$. Recall the definition of $H_{W,V,\mathcal{D}}$ in Section \ref{sec:trace}.
\end{remark}

\begin{proposition}\label{whitenoise-exp}
Let $\{e_i\}_{i\in\mathbb{N}}$ be an orthonormal basis of $L^2_V(\mathbb{T})$. There exists an i.i.d.\ sequence of standard normal random variables $\xi_1, \xi_2, \ldots$ such that for every $h \in L^2_V(\mathbb{T})$,
\begin{equation*}
\dot{B}_V(h) = \sum_{i=0}^\infty \xi_i \<h, e_i\>_V,
\end{equation*}
with convergence in $L^2(\Omega, \mathcal{F}, P)$ and $P$-almost surely.
\end{proposition}

\begin{proof}
$\{\dot{B}_V(e_i)\}$ are independent standard normals. The sum converges in $L^2$ by the isometry property and almost surely by the martingale convergence theorem (see, e.g., \cite[Theorem 4.2.11]{Durrett_2019}). Indeed, if we let $M_n = \sum_{i=0}^n \xi_i \<h, e_i\>_V$, then $(M_n)_{n\geq 1}$ is a martingale such that
$$\sup_{n\in\mathbb{N}} \mathbb{E}|M_n| \leq \sqrt{\mathbb{E}(\dot{B}_V(h)^2)} = \|h\|_V <\infty.$$
\end{proof}

Due to the series expansion representation obtained in Proposition \ref{whitenoise-exp}, we shall use the notation
\begin{equation}\label{eqexpwhite}
\dot{B}_V = \sum_{i=0}^\infty \xi_i e_i
\end{equation}
even though the series on the right-hand side of \eqref{eqexpwhite} does not converge in $ L^2_V(\mathbb{T}) $. However, as we shall see in the following results, the right-hand side of \eqref{eqexpwhite} may converge in a slightly larger space.

\begin{proposition}\label{expwhitedual}
Let $(H, \|\cdot\|_H)$ be a Hilbert space. If $T: L^2_V(\mathbb{T}) \to H$ is Hilbert-Schmidt, then
\begin{equation}\label{expwhitenoiseHS}
T\dot{B}_V := \sum_{i=0}^\infty \xi_i Te_i
\end{equation}
defines an $H$-valued $L^2(\Omega, \mathcal{F}, P)$ random variable, that is, $T\dot{B}_V \in L^2(\Omega, H)$.
\end{proposition}

\begin{proof}
For every $m, n > 0$ we have the estimate 
\begin{equation}\label{convergence_oftheseries}
E\left(\left\|\sum_{n<i\le m} \xi_i Te_i\right\|_H^2 \right)= \sum_{n<i\le m} \|Te_i\|_H^2 \leq \|T\|_{HS}^2 < \infty.
\end{equation}
The inequality in \eqref{convergence_oftheseries} shows that $\sum_{n<i\le m} \xi_i Te_i$ is Cauchy with probability one and consequently defines a well-defined $H$-valued random variable on $\Omega$. Furthermore, by exploiting the upper bound in \eqref{convergence_oftheseries}, it follows that $T\dot{B}_V \in L^2(\Omega, H)$.
\end{proof}

\begin{corollary}\label{whitenoisedual}
Let $\rho \in (0,1/2]$ be as in Theorem \ref{prop6}. If $s > \rho$, then $P$-almost surely, $\dot{B}_V|_{H^s_{W,V}(\mathbb{T})} \in H^{-s}_{W,V}(\mathbb{T})$, and
\begin{equation*}
\dot{B}_V = \sum_{i=0}^\infty \xi_i e_i,
\end{equation*}
where the series on the right-hand side converges in $H^{-s}_{W,V}(\mathbb{T})$.
\end{corollary}

\begin{proof}
By Proposition \ref{HSinclusion}, the inclusion $i : L^2_V(\mathbb{T}) \to H^{-s}_{W,V}(\mathbb{T})$ is trace-class and thus Hilbert-Schmidt. Therefore, by Proposition \ref{expwhitedual},
\begin{equation*}
i(\dot{B}_V) = \sum_{i=0}^\infty \xi_i i(e_i)
\end{equation*}
is a well-defined $L^2(\Omega, \mathcal{F}, P)$ random variable in $H^{-s}_{W,V}(\mathbb{T})$. Furthermore, by Proposition \ref{whitenoise-exp}, we have the following equality for every $h \in H^s_{W,V}(\mathbb{T}) \subset L^2_V(\mathbb{T})$:
\begin{equation*}
\dot{B}_V(h) = i(\dot{B}_V)(h).
\end{equation*}
Therefore, $\dot{B}_V|_{H^s_{W,V}(\mathbb{T})} = i(\dot{B}_V)$.
\end{proof}

We are now prepared to focus our attention on the fractional Matérn-type SPDE of the form
\begin{equation}\label{fracstoch}
(\kappa^2 I - \Delta_{W,V})^\beta u = \dot{B}_V,
\end{equation}
where $ \kappa $ in \eqref{fracstoch} is a well-defined function on $ \mathbb{T} $, bounded and bounded away from zero. We observe that when $ V(x) = W(x) = x $, equation~\eqref{fracstoch} reduces to the Whittle–Matérn equation on the one-dimensional torus (i.e., the circle).

\begin{remark}
Unlike the SPDE in \cite[Section 8]{keale}, \eqref{fracstoch} is driven by $L^2_V(\mathbb{T})$-white noise.
\end{remark}

In what follows, we present the main result of this section concerning the existence and regularity of the sample paths of the solution to \eqref{fracstoch}.

\begin{theorem}
Let $\rho \in (0,1/2]$ be as in Theorem \ref{prop6}. Let $\beta > \rho/2$. Then, the solution to \eqref{fracstoch} is given by
\begin{equation*}
u = (\kappa^2 I - \Delta_{W,V})^{-\beta} \dot{B}_V,
\end{equation*}
which is a centered Gaussian random variable that belongs to $L^2(\Omega, H^{2\beta - \rho - \epsilon}_{W,V}(\mathbb{T}))$ for all $\epsilon > 0$, and has covariance operator $(\kappa^2 I - \Delta_{W,V})^{-2\beta}$.
\end{theorem}

\begin{proof}
Observe that the eigenvalues of $(\kappa^2 I - \Delta_{W,V})^{-\beta}$ are given by $((\kappa^2 + \lambda_i)^{-\beta})_{i\in\mathbb{N}}$, where $(\lambda_i)_{i\in\mathbb{N}}$ are the eigenvalues of $-\Delta_{W,V}$. By Theorem \ref{prop6}, we have that $\kappa^2 + Cn^{1/\rho} \leq \kappa^2 + \lambda_n$ for $n\in\mathbb{N}$ and some constant $C>0$. Since $H_{W,V}^{s_2} \subset H_{W,V}^{s_1}$ for $s_1 < s_2$, it is enough to prove the result for $\epsilon > 0$ such that $2\beta - \rho - \epsilon > 0$. Therefore, since $\beta > \rho/2$, we have
$$\sum_{i=1}^\infty \frac{\lambda_i^{2\beta - \rho - \epsilon}}{(\kappa^2 + \lambda_i)^{2\beta}} \leq \sum_{i=1}^\infty \frac{1}{(\kappa^2 + Cn^{1/\rho})^{\rho + \epsilon}} < \infty,$$
where $\epsilon > 0$ was chosen so that $2\beta > \rho + \epsilon$. 

Therefore, $(\kappa^2 I - \Delta_{W,V})^{-\beta}$ is a Hilbert-Schmidt operator from $L^2_V(\mathbb{T})$ to $H_{W,V}^{2\beta - \rho - \epsilon}(\mathbb{T})$. By Proposition~\ref{expwhitedual}, we have that $u \in L^2(\Omega, H^{2\beta - \rho - \epsilon}(\mathbb{T}))$. Finally, by writing $u$ in terms of the expansion \eqref{expwhitenoiseHS}, we readily obtain that $u$ is a centered Gaussian random field with covariance operator $(\kappa^2 I - \Delta_{W,V})^{-2\beta}$.
\end{proof}

\subsection{Stochastic Partial Differential Equations}
Our goal in this section is to study stochastic partial differential equations (SPDEs) taking values in an analogue of the space of distributions. To this end, we work with the dual of the space $C_{W,V}^\infty(\mathbb{T})$. For simplicity, we set $D_{W,V}(\mathbb{T}) = C^{\infty}_{W,V}(\mathbb{T})$ and denote its strong dual by $D_{W,V}'(\mathbb{T})$. The space $D_{W,V}'(\mathbb{T})$ thus plays the role of the space of distributions on $\mathbb{T}$. Note that when $W(x) = V(x) = x$, the space $D_{W,V}'(\mathbb{T})$ coincides with the classical space of distributions on $\mathbb{T}$.

We are interested in the following parabolic stochastic partial differential equation in $D_{W,V}'(\mathbb{T})$:
\begin{equation}\label{eq:stochastic_partial_differential_equation}
dY_t = \alpha \Delta_{W,V}' Y_t \, dt + \beta \, dN_t,
\end{equation}
where $\alpha, \beta > 0$, $N$ is a mean-zero $D_{W,V}'(\mathbb{T})$-valued martingale, and $\Delta_{W,V}'$ denotes the action of $\Delta_{W,V}$ on $D_{W,V}'(\mathbb{T})$ as given by relation \eqref{eq:dual_action}. This SPDE generalizes (in the one-dimensional case) the equation obtained in \cite{farfansimasvalentim} as the equilibrium fluctuations of an interacting particle system with conductances.

In this section, we use basic results from \cite{kallianpur_1987}, which are summarized in Appendix \ref{app:spde_kallianpur}. We begin with the following result:

\begin{proposition}\label{prop:compatible_family_delta_W_V}
The triple $(C^{\infty}_{W,V}(\mathbb{T}), \Delta_{W,V}, L^2_{V}(\mathbb{T}))$ is a special compatible family (see Appendix \ref{app:spde_kallianpur} for the definition of special compatible families). Furthermore, the restriction of $\Delta_{W,V}$ to $D_{W,V}(\mathbb{T})$ belongs to $L(D_{W,V}(\mathbb{T}))$ and is the infinitesimal generator of a $(C_{0},1)$-semigroup on $D_{W,V}(\mathbb{T})$, which we denote by $P_{W,V}(t)$.
\end{proposition}
\begin{proof}
    The fact that the triple $(C^{\infty}_{W,V}(\mathbb{T}), \Delta_{W,V}, L^2_{V}(\mathbb{T}))$ is a special compatible family follows from Theorem \ref{prop6} and equations \eqref{H2s}, \eqref{formuladual}, and \eqref{intersection}. The remaining claims are consequences of Proposition \ref{prop:special_compatible_family}.
\end{proof}

We now state the following result regarding the existence and uniqueness of solutions to equations of the form \eqref{eq:stochastic_partial_differential_equation}:

\begin{theorem}\label{thm:existence_uniqueness_spde_W_V}
Let $\alpha, \beta > 0$. Let $N = (N_t)_{t \geq 0}$ be a mean-zero $D_{W,V}'(\mathbb{T})$-valued martingale such that $N_0 = 0$ and $E(N_t(f))^2 < \infty$ for all $f \in D_{W,V}(\mathbb{T})$. Let $Y_0$ be a $D_{W,V}'(\mathbb{T})$-valued random variable such that $E\|Y_0\|^2_{-r_0} < \infty$ for some $r_0 \in \mathbb{N}$. Then, the stochastic partial differential equation \eqref{eq:stochastic_partial_differential_equation} admits a unique solution given by
\begin{equation*}
Y_t = P_{W,V}'(t)Y_0 + \int_0^t P_{W,V}'(t-s) \beta \, dN_s,
\end{equation*}
where $P_{W,V}'(t)$ denotes the action of $P_{W,V}(t)$ on $D_{W,V}'(\mathbb{T})$ as given by relation \eqref{eq:dual_action}.
Furthermore, $Y \in D([0,\infty); D_{W,V}'(\mathbb{T}))$ almost surely, and for every $t_0 > 0$ there exists $p_{t_0} > 0$ such that $Y^{t_0} \in D([0,t_0]; D_{W,V}'(\mathbb{T}))$ almost surely and 
$$\mathbb{E}\left(\sup_{0\leq t\leq t_0}\|Y_t\|_{-p_{t_0}}^2\right)<\infty,$$
where $Y^{t_0}$ denotes the restriction of $Y$ to the interval $[0,t_0]$.
\end{theorem}

\begin{proof}
The proof follows from Proposition \ref{prop:compatible_family_delta_W_V} and Proposition \ref{prop:existence_uniqueness_spde}.
\end{proof}

Finally, we have the following result concerning Gaussian solutions:

\begin{corollary}
Let $N = (N_t)_{t \geq 0}$ be a mean-zero $D_{W,V}'(\mathbb{T})$-valued Gaussian martingale with quadratic variation
\begin{equation*}
\langle N(f) \rangle_t = t \int_{\mathbb{T}} \left[ D_W^- f \right]^2 \, dW.
\end{equation*}
Let $Y_0$ be a Gaussian random field in $D_{W,V}'(\mathbb{T})$, independent of $N$, such that $E\|Y_0\|^2_{-r_0} < \infty$ for some $r_0 \in \mathbb{N}$. Then, the stochastic partial differential equation \eqref{eq:stochastic_partial_differential_equation} admits a unique solution.
Furthermore, the solution $Y_t$ is a Gaussian process.
\end{corollary}
\begin{proof}
By Theorem \ref{thm:existence_uniqueness_spde_W_V}, the stochastic partial differential equation \eqref{eq:stochastic_partial_differential_equation} admits a unique solution.
It remains to verify that the solution is Gaussian. By Proposition \ref{prop:gaussianity_solution}, it suffices to check that $N$ has independent increments. However, by L\'evy's characterization of Brownian motion, for each $f \in D_{W,V}(\mathbb{T})$, the process $N(f)$ is a time-changed Brownian motion. Therefore, $N$ has independent increments.
\end{proof}

\begin{remark}
    Observe that it is immediate that if we consider fractional powers $(-\Delta_{W,V})^s$ of the Laplacian $\Delta_{W,V}$ for $s>0$, then $(C^{\infty}_{W,V}(\mathbb{T}), (-\Delta_{W,V})^s, L^2_{V}(\mathbb{T}))$ is also a special compatible family. Therefore, results analogous to those in this section hold for the fractional stochastic partial differential equations
    \begin{equation*}
    dY_t = - \alpha [(-\Delta_{W,V})^s]' Y_t \, dt + \beta \, dN_t,
    \end{equation*}
    with $\alpha, \beta > 0$ and the same conditions on $N$ and $Y_0$ as in Theorem \ref{thm:existence_uniqueness_spde_W_V} or Proposition \ref{prop:gaussianity_solution}, where $[(-\Delta_{W,V})^s]'$ denotes the action of $(-\Delta_{W,V})^s$ on $D_{W,V}'(\mathbb{T})$ as given by relation \eqref{eq:dual_action}.
\end{remark}

% \subsection*{Acknowledgements}
\noindent\textbf{Author Contributions} Kelvin J.R Almeida-Sousa and Alexandre B. Simas contributed equally to this work.\\\\
\textbf{Funding}
The first author acknowledges the support of CAPES (Coordenação de Aperfeiçoamento de Pessoal de Nível Superior), Brazil, through grant No. 88882.440725/2019-01.\\\\
\textbf{Data availability} Not applicable.

\subsection*{Declarations}
\textbf{Conflict of interest} The authors have no Conflict of interest.\\\\
\textbf{Ethics approval } Not applicable.\\\\

\appendix
\section{Stochastic Partial Differential Equations in Nuclear Spaces}\label{app:spde_kallianpur}
In this appendix, we provide a brief overview of basic concepts related to nuclear spaces, as well as some results concerning stochastic differential equations in such spaces.

Let $X$ be a vector space equipped with a sequence of inner products $\langle \cdot, \cdot \rangle_n$ for $n \in \mathbb{N}$, such that the associated norms are increasing: for all $x \in X$ and $n < m$, we have $\|x\|_n \leq \|x\|_m$. Denote by $X_n$ the completion of $X$ with respect to $\|\cdot\|_n$. Define
\begin{equation*}
X_\infty = \bigcap_{n=1}^\infty X_n.
\end{equation*}
The space $(X_\infty, (\|\cdot\|_n)_{n \in \mathbb{N}})$ is called a countably Hilbert space and forms a Fréchet space with respect to the metric
\begin{equation}\label{metric_nuclear}
d(f, g) = \sum_{n=1}^\infty 2^{-n} \frac{\|f - g\|_n}{1 + \|f - g\|_n}.
\end{equation}
Since the norms are increasing, we have $X_m \subset X_n$ for all $m \geq n$.

A countably Hilbert space $X_\infty$ is called \emph{nuclear} if, for each $n \geq 0$, there exists $m > n$ such that the natural inclusion $i_{m,n}: X_m \to X_n$ is a Hilbert--Schmidt operator. That is, for any orthonormal basis $\{e_j\}_{j \geq 1}$ of $X_m$, we have
\begin{equation*}
\sum_{j=1}^\infty \|e_j\|_n^2 < \infty.
\end{equation*}

Let $L(X_\infty)$ denote the space of continuous linear operators from $X_\infty$ to itself. Let $X_{-n}$ denote the Hilbert space dual of $X_n$. For any $y \in X_{-n}$, define
\begin{equation*}
\|y\|_{-n} = \sup_{\substack{\|x\|_n \leq 1\\ x\in X_n}} |y(x)|.
\end{equation*}
We have
\begin{equation*}
X_{-n} \subset X_{-m} \quad \text{for all } m \geq n.
\end{equation*}

Let $ X_{-\infty} $ be the topological dual of $ X_{\infty} $ with respect to the strong topology, which is defined by the system of neighborhoods of zero of the form $ \{ y \in X_{-\infty} : \|y\|_{-\infty,B} < \epsilon \} $, where
\begin{equation*}
\|y\|_{-\infty,B} = \sup\{ |y(x)| : x \in B \},
\end{equation*}
and $ B $ is a bounded subset of $X_{\infty}$. That is, for every neighborhood $V$ of zero in the topology induced by the metric $d(\cdot,\cdot)$ given in \eqref{metric_nuclear}, there exists a constant $r > 0$ such that $B \subset r V$. Thus,
\begin{equation*}
X_{-\infty} = \bigcup_{r=1}^{\infty} X_{-r}.
\end{equation*}

Let $ T \in L(X_\infty) $. The operator $T$ induces a continuous linear operator on $X_{-\infty}$, denoted by $T'$, which is defined as follows: for $y \in X_{-\infty}$ and $x \in X_\infty$,
\begin{equation}\label{eq:dual_action}
    (T'y)(x) := y(Tx).
\end{equation}

We now define the notions of a $ (C_{0},1) $-semigroup and its infinitesimal generator on a countably Hilbert nuclear space.

\begin{definition}
A family of linear operators $ \{S(t) : t \geq 0\} $ on $ X_\infty $ is called a $ (C_{0},1) $-semigroup if:
\begin{enumerate}
    \item $ S(t_1)S(t_2) = S(t_1 + t_2) $ for all $ t_1, t_2 \geq 0 $, and $ S(0) = I $;
    \item The mapping $ t \mapsto S(t)x $ is continuous in $ X_\infty $ for each $ x \in X_\infty $;
    \item For every $ q \geq 0 $, there exist constants $ M_q > 0 $, $ \sigma_q > 0 $, and $ p \geq q $ such that
    \begin{equation*}
    \|S(t)x\|_q \leq M_q e^{\sigma_q t} \|x\|_p, \quad \forall x \in X_\infty,\ t > 0.
    \end{equation*}
\end{enumerate}
\end{definition}

\begin{definition}
Let $\{S(t) : t \geq 0\}$ be a $(C_0,1)$-semigroup on $X_\infty$. The infinitesimal generator $A$ of $S(t)$ is defined by
\begin{equation*}
A x = \lim_{s \downarrow 0} \frac{S(s)x - x}{s}
\end{equation*}
whenever the limit exists in $X_\infty$. The domain of $A$, denoted $D(A)$, is the set of all $x \in X_\infty$ for which the above limit exists.
\end{definition}

\begin{definition}\label{def:special_compatible_family}
A triple $(X_{\infty}, L, H)$ is called a special compatible family if the following conditions hold:
\begin{enumerate}
    \item $H$ is a real separable Hilbert space with inner product $\langle \cdot, \cdot \rangle_H$;
    \item $L$ is a closed, densely defined, self-adjoint operator on $H$ with $\langle -L \varphi, \varphi \rangle_H \geq 0$ for all $\varphi \in D(L)$;
    \item There exists $r > 0$ such that $(I - L)^{-r}$ is Hilbert--Schmidt;
    \item $X_{\infty}$ is defined as
    \begin{equation*}
    X_{\infty} = \left\{ \varphi \in H : \| (I - L)^r \varphi \|_H^2 < \infty \text{ for all } r \in \mathbb{R} \right\},
    \end{equation*}
    with inner product
    \begin{equation*}
    \langle \varphi, \psi \rangle_r = \sum_{j=1}^\infty (1 + \lambda_j)^{2r} \langle \varphi, \varphi_j \rangle_H \langle \psi, \varphi_j \rangle_H,
    \end{equation*}
    where $\{\varphi_j\}_{j \geq 1}$ is a complete orthonormal system of eigenvectors of $-L$ with $-L \varphi_j = \lambda_j \varphi_j$.
\end{enumerate}
\end{definition}

The following proposition summarizes some results proved in \cite[Section 3]{kallianpur_1987}:

\begin{proposition}\label{prop:special_compatible_family}
Suppose $(X_{\infty}, L, H)$ is a special compatible family in the sense of Definition \ref{def:special_compatible_family}. Then, the restriction of $L$ to $X_{\infty}$ (also denoted by $L$) belongs to $L(X_{\infty})$, is the infinitesimal generator of a $(C_{0},1)$-semigroup on $X_{\infty}$, and $X_{\infty}$ is a nuclear space.
\end{proposition}

Throughout this section, let $(\Omega, \mathcal{F}, (\mathcal{F}_t)_{t \geq 0}, P)$ be a filtered probability space satisfying the usual conditions; that is, the filtration is $P$-complete and right-continuous.

\begin{definition}
A stochastic process $M = (M_t)_{t \geq 0}$ is called an $X_{-\infty}$-valued $\mathcal{F}_t$-martingale if, for every $y \in X_{\infty}$, the real-valued process $(M_t(y))_{t \geq 0}$ is an $\mathcal{F}_t$-martingale. We say that $M$ is a Gaussian martingale if, for every $y \in X_{\infty}$, the real-valued process $(M_t(y))_{t \geq 0}$ is a Gaussian process.
\end{definition}

Let $ M = (M_t)_{t \geq 0} $ be an $ X_{-\infty} $-valued $\mathcal{F}_t$-martingale with $ M_0 = 0 $. We are interested in the existence and uniqueness of solutions to the stochastic evolution equation:
\begin{equation}\label{eq:stochastic_evolution}
    d\xi_t = A' \xi_t \, dt + dM_t, \qquad \xi_0 = \gamma,
\end{equation}
where $ \gamma $ is an $ X_{-\infty} $-valued random variable.

\begin{definition}
A process $ \xi = (\xi_t)_{t \geq 0} $ is called an $ X_{-\infty} $-valued solution to equation~\eqref{eq:stochastic_evolution} if:
\begin{enumerate}
    \item $ \xi_t $ is $ X_{-\infty} $-valued, progressively measurable, and $\mathcal{F}_t$-adapted;
    \item For all $ y \in X_{-\infty} $ and $ t \geq 0 $,
    \begin{equation*}
    \xi_t(y) = \gamma(y) + \int_0^t \xi_s(Ay) \, ds + M_t(y).
    \end{equation*}
\end{enumerate}
\end{definition}

We denote by $D(T; W)$ the space of right-continuous processes with left limits (càdlàg processes) on $T$ with values in the space $W$. Also, for a stochastic process $Z=(Z_t)_{t\geq 0}$ and $t_0>0$, we denote by $Z^{t_0}=(Z_t)_{t\in [0,t_0]}$ the restriction of $Z$ to the interval $[0,t_0]$.

\begin{proposition}[Existence and Uniqueness]\label{prop:existence_uniqueness_spde}
Suppose the following hold:
\begin{enumerate}
    \item $ \gamma $ is an $ X_{-\infty} $-valued, $ \mathcal{F}_0 $-measurable random variable such that for some $ r_0 \in\mathbb{N} $, $ \mathbb{E}\|\gamma\|^2_{-r_0} < \infty $;
    \item $ M = (M_t)_{t \geq 0} $ is an $ X_{-\infty} $-valued martingale with $ M_0 = 0 $ and, for each $ y \in X_{-\infty} $, $ \mathbb{E}(M_t(y))^2 < \infty $;
    \item $ A $ is a continuous linear operator on $ X_\infty $ and is the infinitesimal generator of a $ (C_{0},1) $-semigroup $ \{S(t)\} $ on $ X_\infty $.
\end{enumerate}
Then, the stochastic evolution equation above admits a unique solution given by
\begin{equation*}
\xi_t = S'(t)\gamma + \int_0^t S'(t-s) \, dM_s.
\end{equation*}
Furthermore, $\xi \in D([0,\infty); X_{-\infty})$ almost surely, and for every $t_0>0$ there exists $p_{t_0}>0$ such that $\xi^{t_0} \in D([0,t_0]; X_{-p_{t_0}})$ almost surely, and 
\begin{equation*}
\mathbb{E}\left(\sup_{0\leq t\leq t_0}\|\xi_t\|_{X_{-p_{t_0}}}^2\right)<\infty.
\end{equation*}
\end{proposition}

\begin{remark}
Condition 2 is satisfied if $ \mathbb{E}(M_t(y))^2 = t Q(y, y) $, where $ Q $ is a positive definite, continuous bilinear form on $ X_{-\infty} \times X_{-\infty} $.
\end{remark}

\begin{proposition}[Gaussianity of the Solution]\label{prop:gaussianity_solution}
Let $M$ be a Gaussian martingale with independent increments. Also, let $ \gamma $ be a Gaussian random variable in $ X_{-\infty} $ such that $ \mathbb{E}\|\gamma\|^2_{-r_0} < \infty $ for some $ r_0 \in\mathbb{N} $, and assume that $ \gamma $ is independent of the martingale $M$. Then, the solution $ \xi = (\xi_t) $ is a Gaussian process in $ X_{-\infty} $.
\end{proposition}

%% The Appendices part is started with the command \appendix;
%% appendix sections are then done as normal sections
%% \appendix

%% \section{}
%% \label{}

%% If you have bibdatabase file and want bibtex to generate the
%% bibitems, please use
%%
% \bibliographystyle{elsarticle-harv} 

\bibliography{sn-bibliography}

%% else use the following coding to input the bibitems directly in the
%% TeX file.

% \begin{thebibliography}{00}

% %% \bibitem{label}
% %% Text of bibliographic item

% \bibitem{references}

% \end{thebibliography}
\end{document}